\documentclass[11pt,a4paper,twoside]{article}

\usepackage[a4paper,inner=30mm,outer=30mm,top=37.5mm,bottom=37.5mm]{geometry}
\usepackage{setspace}
\usepackage[T1]{fontenc}
\usepackage[utf8]{inputenc}
\usepackage{amsmath,amsthm,amssymb,amsfonts,amstext,amsopn,mathrsfs,stmaryrd}
\usepackage{mathtools}
\usepackage{CJKutf8}
\usepackage{etoolbox}
\usepackage{pgfplots}
\pgfplotsset{compat=1.18}
\usepackage{booktabs}
\AtBeginEnvironment{thebibliography}{\begin{CJK}{UTF8}{gbsn}}
\AtEndEnvironment{thebibliography}{\end{CJK}}
\numberwithin{equation}{section}

\theoremstyle{plain}
\newtheorem{thm}{Theorem}[section]
\newtheorem{prop}{Proposition}[section]
\newtheorem{lem}[thm]{Lemma}

\theoremstyle{definition}

\theoremstyle{remark}

\usepackage{graphicx}
\usepackage{xcolor}
\usepackage{booktabs,threeparttable,multirow}
\usepackage{array,tabularx,makecell}

\newcolumntype{L}[1]{>{\raggedright\arraybackslash}p{#1}}
\newcolumntype{Y}{>{\raggedright\arraybackslash}X}
\usepackage{subcaption}
\usepackage{algorithm}
\usepackage{algpseudocode}
\usepackage[authoryear,round]{natbib}
\usepackage{authblk}
\usepackage{listings}
\usepackage{tcolorbox}
\usepackage{dsfont}
\let\epsilon\varepsilon

\newcommand{\E}{\mathbb{E}}
\newcommand{\PP}{\mathbb{P}}
\newcommand{\1}{\mathds{1}}

\newcommand{\dist}{\operatorname{dist}}
\newcommand{\Var}{\operatorname{Var}}
\newcommand{\Cov}{\operatorname{Cov}}

\newcommand{\TV}{\operatorname{TV}}
\newcommand{\HS}{\operatorname{HS}}

\newcommand{\den}{\mathrm{den}}

\newcommand{\test}{\mathrm{test}}
\newcommand{\est}{\mathrm{est}}
\newcommand{\sep}{\mathrm{sep}}

\newcommand{\op}{\mathrm{op}}

\newcommand{\calI}{\mathcal I}
\newcommand{\calV}{\mathcal V}

\usepackage{xurl}
\usepackage{enumerate}
\usepackage[colorlinks=true,linkcolor=teal,citecolor=brown,urlcolor=magenta]{hyperref}
\usepackage{longtable}

\newif\ifshowcomments
\showcommentstrue

\title{Adaptive Confidence Sets for Nonparametric Density Estimation under Local Differential Privacy}
\author[1]{Aditya Khan\thanks{Corresponding author: \href{mailto:aditya.khan@columbia.edu}{aditya.khan@columbia.edu}}}
\author[2]{Thibault Randrianarisoa}
\affil[1]{Department of Statistics, Columbia University}
\affil[2]{Department of Statistical Sciences, University of Toronto}

\date{}

\begin{document}
\maketitle

\begin{abstract}
In nonparametric density estimation without privacy constraints, it is known that a confidence set whose coverage holds uniformly over
a `rough' class of regularity $r\geq0$ can simultaneously attain, in $L^2$-diameter, the minimax rate over
a smaller, smoother class of regularity $s>r$ if and only if $s \le 2r$. We study such adaptive honest
confidence sets under
non-interactive $\alpha$-local differential privacy (LDP) constraints.
Our main result is that, for bounded $\alpha$, such `smooth-rate adaptation' is now possible if and only if $s \le 4r/3$. In this smaller range we construct an honest
private confidence set whose diameter is of the private minimax order on each of the two
classes. As
in the non-private case, the threshold arises because an honest set cannot, at a smooth
density, be narrower than the separation rate arising in a related testing problem. Under LDP we prove that this private testing rate is of order $(n\alpha^2)^{-2r/(4r+3)}$, which is at
most the private estimation rate $(n\alpha^2)^{-s/(2s+2)}$ over the smooth class exactly
when $s \le 4r/3$. For $s > 4r/3$, however, we show that smooth-rate adaptation becomes possible once one discards from the rough class the densities that lie outside
the smooth class but within an $L^2$-radius of the order of
the private testing rate.
\end{abstract}

\section{Introduction}
Local differential privacy \citep{warnerRandomizedResponseSurvey1965,kasiviswanathanWhatCanWe2011,DJW18} describes the situation in which each individual passes their own observation through a known randomising kernel before releasing it, so that a statistician only ever observes privatised reports. With no trusted curator, it is the most demanding of the usual trust models of differential privacy \citep{Dwork06b,Dwork06,dworkAlgorithmicFoundationsDifferential2014}, and this protection has a statistical price. Much of the recent literature has been devoted to quantifying it for estimation and testing, while the price of privacy for uncertainty quantification is far less well understood. We study this price in a canonical problem of nonparametric inference: constructing, from locally privatised data, honest confidence sets for a density of unknown smoothness.

A confidence set $C_n$ for the density $f_0$ of the observations is \emph{honest} \citep{liHonestConfidenceRegions1989} over a model $\mathcal{F}$ at level $1-\gamma$ if $\inf_{f\in\mathcal{F}}P_f(f\in C_n)\ge 1-\gamma$ where $P_f$ denotes the law of the data under density $f$, so that its coverage does not deteriorate as the truth varies within the model. Honesty alone is achieved trivially by $C_n=\mathcal{F}$, so the question of interest is whether the $L^2$-diameter of an honest set can, in addition, shrink at the minimax estimation rate of whichever smaller class $\mathcal{F}_0\subset\mathcal{F}$ contains $f_0$, without it being known. For estimation, adaptation to unknown smoothness is available in great generality, wavelet thresholding being the standard example in density estimation \citep{donohoDensityEstimationWavelet1996,juditsky2004minimax}, and an adaptive estimator need only achieve the minimax rate of the smallest class containing the true density, without having to quantify its own error. A confidence set, by contrast, must quantify this error: its coverage must hold over the whole model, while its diameter should shrink at the rate of the class containing $f_0$. This requires deciding from the data whether $f_0$ lies in the smoother class, so that adaptive confidence sets are governed by a testing problem rather than by an estimation problem \citep{lowNonparametricConfidenceIntervals1997,baraud2004confidence,CaiLow2006,robinsAdaptiveNonparametricConfidence2006,gineMathematicalFoundationsInfiniteDimensional2016}.

Let us make this precise in the setting of the present paper, for the $L^2$-diameter and a pair of Besov-type classes of regularities $0<r<s$. Without privacy, the smooth class can be estimated at rate $n^{-s/(2s+1)}$, but rough densities can be distinguished from it only at $L^2$-distances from the smooth class beyond the separation rate $n^{-2r/(4r+1)}$ \citep{ermakovAsymptoticallyMinimaxTests1990, ingsterNonparametricGoodnessofFitTesting2003,bullAdaptiveConfidenceSets2012}. If an honest set had a diameter of smaller order than this separation rate on the smooth class, it would yield a test separating that class from rough densities closer to it than the separation rate, which is impossible. Hence smooth-rate adaptation requires $n^{-s/(2s+1)}\gtrsim n^{-2r/(4r+1)}$, that is $s\le 2r$, a condition that is also sufficient \citep{bullAdaptiveConfidenceSets2012,gineMathematicalFoundationsInfiniteDimensional2016}. Analogous phenomena arise for other metrics and models \citep{genoveseConfidenceSetsNonparametric2005,genoveseAdaptiveConfidenceBands2008,hoffmannAdaptiveInferenceConfidence2011,carpentier2013honest,deoAdaptiveConfidenceSets2021}. When $s>2r$, adaptation becomes possible again once one removes from the rough class a `critical region' of densities lying within the separation rate of the smooth class \citep{bullAdaptiveConfidenceSets2012}.


The minimax theory of locally private estimation was developed by \citet{DJW13,duchiLocalPrivacyStatistical2013,DJW18}, who derived private analogues of the Le~Cam, Fano and Assouad methods. For univariate densities in Sobolev or Besov balls of regularity $t$, the minimax $L^2$-rate under $\alpha$-LDP, where $\alpha>0$ is the privacy level, is $n^{-t/(2t+1)}\vee(n\alpha^2)^{-t/(2t+2)}$ \citep{DJW13,DJW18,Butucea20}. Privacy thus alters the exponent of the rate, whereas in regular parametric models the $n^{-1/2}$ rate is preserved and only the constant deteriorates \citep{DJW18,steinberger2024efficiency,kalinin2025efficient}. Adaptation to unknown $t$ costs a logarithmic factor \citep{Butucea20}, which \citet{cai2025cost} have shown to be unavoidable, and adaptive procedures also exist for pointwise \citep{Kro21} and spectral \citep{kroll2020adaptive} density estimation. For testing, \citet{lam-weilMinimaxOptimalGoodnessoffit2021} obtained the separation rate $(n\alpha^2)^{-2t/(4t+3)}$ for goodness-of-fit testing of a simple null over Besov balls under non-interactive LDP. \citet{acharyaTestTrustOptimal2019} and \citet{BerrettButucea20} treated discrete distributions and H\"older densities were considered by \citet{Dubois23}, the last two showing that sequentially interactive mechanisms improve the separation rates. Locally private density estimation has also been studied for multivariate densities \citep{GK21}, under Hellinger loss \citep{Sart23} and in an adversarial framework \citep{AlbertAdversarialDP}, as have functional estimation \citep{RS20,ButuceaIssartel21,ButuceaQuadraticDP23,randrianarisoaMultipurposeLocallyDifferentiallyprivate2025}, regression and classification \citep{BGW21,gyorfi2022rate}.

\paragraph{Uncertainty quantification under privacy.}
In the central version of differential privacy, \citet{karwa2018finite} constructed finite-sample confidence intervals for a Gaussian mean together with lower bounds on their length, and valid intervals have since been obtained for medians \citep{drechsler2022nonparametric} and for $M$-estimators \citep{avella2023differentially}. General-purpose devices have also been transposed to privatised data: the parametric \citep{ferrando2022parametric} and nonparametric \citep{covington2025unbiased,chadha2024resampling,wang2025differentially} bootstrap, subsampling schemes that turn any private estimator into an interval \citep{shoham2025differentially}, `repro sample' inference \citep{awan2024simulation}, Bayesian inference by data augmentation \citep{ju2022data} and conformal prediction \citep{angelopoulos2022private}. In the local model, confidence intervals are available for Gaussian means \citep{gaboardi2019locally}, bounded means in nonparametric settings \citep{waudby2023nonparametric}, and conformal prediction sets \citep{penso2025privacy,zhang2025online}. In all of these works the target is a finite-dimensional quantity, and the width of the set is not required to adapt to an unknown characteristic of the sampling distribution, e.g. its smoothness. To the best of our knowledge, honest adaptive confidence sets in a nonparametric model have not been studied under any form of privacy.

Our contributions are threefold.
\begin{enumerate}
\item \textit{Private adaptation boundary.} Under bounded local privacy, we prove that honest adaptive $L^2$ confidence sets exist if and only if $s \le 4r/3$.

\item \textit{Composite smoothness-testing lower bound.} We extend the locally private goodness-of-fit lower-bound strategy of \citet{lam-weilMinimaxOptimalGoodnessoffit2021} to a composite null.

\item \textit{Adaptation under separation.} For $s > 4r/3$, we construct adaptive confidence sets after restricting the rough class to densities separated from the smooth class by a margin proportional to the private testing rate, up to logarithmic factors.
\end{enumerate}

The change in the adaptation boundary is driven by the unequal effect of local privacy on estimation and smoothness testing. Local privacy slows smoothness testing more strongly than estimation, changing the boundary from $s \le 2r$ to $s \le 4r/3$.


\section{Model and assumptions}
\label{sec:model}

\paragraph{Locally private observations.}
Let $X_1,\ldots,X_n$ be independent with an unknown density $f$ on $[0,1]$ and distribution $P_f$. Only reports $Z_1,\ldots,Z_n$ are observed: individual $i$ draws $Z_i$ from a known Markov kernel $Q_i(\cdot\mid X_i)$, with its own randomisation, so that
$$Q^n(dz_1,\ldots,dz_n\mid x_1,\ldots,x_n)=\prod_{i=1}^nQ_i(dz_i\mid x_i).$$
Such channels are called \textit{non-interactive}. Their output spaces have countably generated $\sigma$-fields, as $\mathbb R^d$ or finite sets do. The channel $Q_i$ is $\alpha$-locally differentially private ($\alpha$-LDP) if $Q_i(A\mid x)\leq e^\alpha Q_i(A\mid x')$ for every measurable $A$ and all $x,x'\in[0,1]$. We write $P_f^{Q^n}$ and $\E_f^{Q^n}$ for the law and the expectation of the reports.

\paragraph{Besov classes.}
We use the boundary-corrected wavelet basis $\{\varphi_{j_0k},\psi_{jk}\}$ of $L^2([0,1])$ of Appendix~\ref{app: bes_wav}, whose regularity exceeds a fixed $S>0$ and whose scaling functions span the constants. Write $\beta_j(f)=(\langle f,\psi_{jk}\rangle)_{k\in N_j}$ for the wavelet coefficients of $f$ at level $j\geq j_0$, and let $K_J$ be the orthogonal projection onto the span of the scaling functions and of the wavelets at levels $j_0,\ldots,J-1$. For $0<t<S$, the Besov norm is
$$\|f\|_{B^t_{2,\infty}}=\max\Bigl\{\|K_{j_0}f\|_2,\ \sup_{j\geq j_0}2^{jt}\|\beta_j(f)\|_{\ell^2}\Bigr\},$$
and, for $R\geq1$ and $U>1$, 
\[\Sigma(t,R,U)=\left\{f\ge0:\begin{array}{l}\int_0^1f(x)\,dx=1,\\ \|f\|_{B^t_{2,\infty}}\le R,\quad \|f\|_\infty\le U\end{array}\right\}.\]
These classes are nested, $\Sigma(s,R,U)\subseteq\Sigma(r,R,U)$ for $r<s$, and all contain the uniform density. For $f\in\Sigma(t,R,U)$ and $J>j_0$, the squared approximation error satisfies the deterministic bound $\|(I-K_J)f\|_2^2\leq b_t(J)=R^22^{-2Jt}/(1-2^{-2t})$ (Lemma~\ref{lem:besov-tail}), which plays the role of a squared bias below.

\paragraph{Benchmark rates.}
For $t,r>0$, put
$$
\begin{aligned}
\delta_{n,\est}(t,\alpha)&\coloneqq n^{-t/(2t+1)}\vee(n\alpha^2)^{-t/(2t+2)},\\
\rho_{n,\test}(r,\alpha)&\coloneqq n^{-2r/(4r+1)}\vee(n\alpha^2)^{-2r/(4r+3)}.
\end{aligned}
$$
The first is the minimax rate for estimating $f$ in $L^2$-norm over such classes, without and with local privacy \citep{DJW13,DJW18,Butucea20}. The second is the separation rate for testing smoothness (Section~\ref{sec:smoothness-testing}). Under bounded privacy, $\alpha\leq\bar\alpha$, the private terms dominate: $\delta_{n,\est}(t,\alpha)\asymp(n\alpha^2)^{-t/(2t+2)}$ and $\rho_{n,\test}(r,\alpha)\asymp(n\alpha^2)^{-2r/(4r+3)}$. The private rates terms above are then slower than their non-private counterparts. Finally, $\ell_n=\{\log(e+n\alpha^2)\}^{\nu+1/2}$, where $\nu>1$ is the exponent of the noise allocation of the privacy mechanism from Appendix~\ref{sec: privacy release}, adapted from the \citet{Butucea20}, is the logarithmic price of the adaptive estimator of Proposition~\ref{prop:adaptive-private-estimator} used in our constructions.

\paragraph{Honest and adaptive confidence sets.}
A confidence set $C_n=C_n(Z_1,\ldots,Z_n)$ is honest over a class $\mathcal F$ at level $1-\gamma$ if
$$\inf_{f\in\mathcal F}P_f^{Q^n}(f\in C_n)\geq1-\gamma.$$
Its $L^2$-diameter is $|C_n|_2=\sup_{g,h\in C_n}\|g-h\|_2$, with $|\varnothing|_2=0$, and we assume that the events involving $C_n$ are measurable. For $\delta_n\to0$, we say that $C_n$ has diameter of order $\delta_n$ over $\Sigma(t,R,U)$ if \[\sup_{f\in\Sigma(t,R,U)}P_f^{Q^n}(|C_n|_2>M_n\delta_n)\to0\] for every $M_n\to\infty$. Given a set $\mathcal R\subset(0,S)$ of regularities with smallest element $r$, we require honesty over the rough class $\Sigma(r,R,U)$, which contains $\Sigma(t,R,U)$ for every $t\in\mathcal R$, and ask whether the diameter can nevertheless be of the order of the estimation rate over $\Sigma(t,R,U)$, simultaneously for every $t\in\mathcal R$. For the impossibility results below, two regularities $r<s$ suffice, and $\Sigma(s,R,U)$ is then called the smooth class.

\paragraph{Standing assumptions.}
Throughout, all regularities lie in $(0,S)$, the constants $R\geq1$, $U>1$, $\nu>1$ and $\gamma\in(0,1)$ are fixed, and privacy is bounded, $0<\alpha\leq\bar\alpha$, which covers the levels used in practice\footnote{See \url{https://desfontain.es/blog/real-world-differential-privacy.html} for privacy levels used in real-world deployments of differential privacy.} \citep[e.g.][]{DJW18}. The level $\alpha$ may depend on $n$, as long as $n\alpha^2\to\infty$. Constants may depend on the regularities involved, on $R$, $U$, $S$, $\nu$, $\bar\alpha$, $\gamma$ and on the basis, but not on $n$ or $\alpha$.

\section{The adaptation boundary}
\label{sec:adaptation-boundary}

Our main result can be summarised as follows.

\begin{thm}
\label{thm:main-adaptation}
Let $\mathcal R\subset(0,S)$ be a set of regularities with smallest element $r$. Under the standing assumptions, the following statements hold.
\begin{enumerate}
\item If $\mathcal R\subseteq[r,4r/3]$, there exist confidence sets based on non-interactive $\alpha$-LDP reports that are honest over $\Sigma(r,R,U)$ at level $1-\gamma$, for every $n$, and whose diameter is of order $\ell_n\delta_{n,\est}(t,\alpha)$ over $\Sigma(t,R,U)$, for every $t\in\mathcal R$.
\item If $\mathcal R$ contains some $s>4r/3$ and $\gamma<1/2$, no confidence sets based on non-interactive $\alpha$-LDP reports can be honest over $\Sigma(r,R,U)$ at level $1-\gamma$ and have diameter of order $\ell_n\delta_{n,\est}(s,\alpha)$ over $\Sigma(s,R,U)$.
\end{enumerate}
\end{thm}

Adaptation is thus possible exactly over ranges whose largest regularity is at most $4/3$ times the smallest, and the sets of Section~\ref{sec:construction} achieve it while depending on $\mathcal R$ only through $r$. The second statement involves only $r$ and $s$: a single class beyond the boundary rules out adaptation.

Without privacy, the corresponding range is $[r,2r]$ \citep{bullAdaptiveConfidenceSets2012,gineMathematicalFoundationsInfiniteDimensional2016}. In both cases, the boundary comes from comparing the two benchmark rates: a set that adapts to $\Sigma(t,R,U)$ induces a test of $\Sigma(t,R,U)$ against rough densities at a distance of the order of its diameter over $\Sigma(t,R,U)$ (Section~\ref{sec:smoothness-testing}), so that adaptation to $t$ requires $\delta_{n,\est}(t,\alpha)\gtrsim\rho_{n,\test}(r,\alpha)$, up to logarithmic factors. Under bounded privacy, this reads $t/(2t+2)\leq2r/(4r+3)$, that is, $t\leq4r/3$. Without privacy, it reads $t/(2t+1)\leq2r/(4r+1)$, that is, $t\leq2r$. The pictures changes as the two rates, for estimation and testing, react differently to privacy.

The factor $\ell_n$ above is the price for adaptation incurred by the estimator used in the construction. It does not affect the boundary as it is of polynomial order, and the bounds steps from polynomial discrepancy.

\section{Smoothness testing}
\label{sec:smoothness-testing}

The impossibility result in Theorem \ref{thm:main-adaptation} involves two regularities $r<s$, which we fix in this section. Let $\dist_2(f,A)=\inf_{g\in A}\|f-g\|_2$ for $A\subseteq L^2([0,1])$. For $\rho>0$, we test $H_0:f\in\Sigma(s,R,U)$ against $f\in H_1(\rho)$, the set of densities $f\in\Sigma(r,R,U)$ with $\dist_2(f,\Sigma(s,R,U))\geq\rho$. The null is composite and the alternatives are far from every smooth density.

\begin{prop}
\label{thm:testing-lower}
Under the standing assumptions, let $\rho_n=o\{\rho_{n,\test}(r,\alpha)\}$. Then, for every sequence of non-interactive $\alpha$-LDP channels and every sequence of tests $\Psi_n=\Psi_n(Z_1,\ldots,Z_n)$,
$$
\begin{aligned}
\liminf_{n\to\infty}\Big[&\sup_{f\in\Sigma(s,R,U)}P_f^{Q^n}(\Psi_n=1)\\
&+\sup_{f\in H_1(\rho_n)}P_f^{Q^n}(\Psi_n=0)\Big]\geq1.
\end{aligned}
$$
\end{prop}

The proof, in Appendix~\ref{app:testing-lower-impossibility}, follows the standard route for testing lower bounds \citep{ingsterNonparametricGoodnessofFitTesting2003}. One constructs a probability distribution, or prior, on the alternatives, and shows that the law of the reports when $f$ is drawn from this prior is close in total variation to their law under the uniform density, which belongs to the null. No test can then have small errors under both hypotheses. We adapt here the channel-dependent prior that \citet{lam-weilMinimaxOptimalGoodnessoffit2021} use for a simple null to our composite null, as in the non-private composite lower bounds of \citet{bullAdaptiveConfidenceSets2012} and \citet{carpentierTestingRegularitySmooth2015}.

The link with confidence sets is then standard \citep{gineMathematicalFoundationsInfiniteDimensional2016}. Let $C_n$ be honest over $\Sigma(r,R,U)$ at level $1-\gamma$, with diameter of order $\ell_n\delta_{n,\est}(s,\alpha)$ over $\Sigma(s,R,U)$, and put $v_n=M_n\ell_n\delta_{n,\est}(s,\alpha)$ for a deterministic sequence $M_n\to\infty$. By definition, $P_f^{Q^n}(|C_n|_2>v_n)\to0$ uniformly over $\Sigma(s,R,U)$, so that $v_n$ is a slightly inflated version of the smooth diameter. Consider the test
\[\Psi_n=\1\{C_n\cap\Sigma(s,R,U)=\varnothing\ \text{or}\ |C_n|_2>v_n\}.\]
Under the null, it rejects only if $C_n$ fails to cover $f$ or is larger than $v_n$, which has probability at most $\gamma+o(1)$. At a density $f\in H_1(2v_n)$, it does not reject only if $C_n$ fails to cover $f$: otherwise, $C_n$ would contain $f$ together with a smooth density within distance $|C_n|_2\leq v_n$ of $f$, whereas $f$ is at distance at least $2v_n$ from the smooth class (Proposition~\ref{prop:testing-reduction}). For $\gamma<1/2$, the two errors therefore sum to at most $2\gamma+o(1)<1$ at the separation $2v_n$. When $s>4r/3$, the ratio $\ell_n\delta_{n,\est}(s,\alpha)/\rho_{n,\test}(r,\alpha)$ tends to zero polynomially fast, so that $M_n$ can be chosen to grow slowly enough for $2v_n=o\{\rho_{n,\test}(r,\alpha)\}$, which contradicts Theorem~\ref{thm:testing-lower}. This proves the second statement of Theorem~\ref{thm:main-adaptation}.

\section{The construction}
\label{sec:construction}

The confidence set behind the first statement of Theorem~\ref{thm:main-adaptation} depends on $\mathcal R$ only through its smallest element $r$, and its construction involves two independent sample splits. We partition the sample deterministically into groups $I_1$ and $I_2$ of sizes $m_1=\lfloor n/2\rfloor$ and $m_2=n-m_1$. The set is centred at an adaptive private estimator computed on $I_1$. Its radius is built on $I_2$, given this estimator, by bounding the squared distance between $f$ and the centre in two parts. The part carried by the levels above a resolution $J$ is at most $b_r(J)$ for every density of the rough class (see Section \ref{sec:model}). The part below $J$ is estimated without bias from the second split, and the estimate is enlarged by a margin that accounts for its variance, including the variance due to the privacy noise. The radius is the sum of the two bounds.

 For $J>j_0$, let $\calI_J$ index the wavelets at levels below $J$, denote these functions by $\psi_a$, $a\in\calI_J$, and put $\theta_J(h)=(\langle h,\psi_a\rangle)_{a\in\calI_J}$. The private release $Z_{i,J}$ of Appendix~\ref{sec: privacy release} adds to $(\psi_a(X_i))_{a\in\calI_J}$ independent Laplace noise with scales $\sigma_{J,a}$ satisfying $\sigma_{J,a}^2\asymp\alpha^{-2}(j\vee1)^{2\nu}2^j$ for a coefficient at level $j$: finer levels receive more noise, and the factor $(j\vee1)^{2\nu}$ spreads the privacy budget across levels. On $I_1$, we compute an adaptive wavelet-thresholding estimator $\widetilde f_{m_1}$ attaining the minimax rate of convergence for any regularity $0<t<S$, up to a logarithmic factor (see Proposition~\ref{prop:adaptive-private-estimator}).

Given $\widetilde f_{m_1}$ and $J>j_0$, the diagonal-free U-statistic
\[
\begin{aligned}
U_{m_2}(\widetilde f_{m_1},J)&=\frac{1}{m_2(m_2-1)}\sum_{\substack{i,i'\in I_2\\i\ne i'}}\\
&\quad\times\langle Z_{i,J}-\theta_J(\widetilde f_{m_1}),Z_{i',J}-\theta_J(\widetilde f_{m_1})\rangle
\end{aligned}
\]
is, given $I_1$, unbiased for the projected loss $\|K_J(f-\widetilde f_{m_1})\|_2^2$ (Lemma~\ref{lem:projected-loss-moments}). Its variance is controlled through
$$\mathsf V_J=U(I-e_Je_J^\top)+2\operatorname{diag}_{a\in\calI_J}(\sigma_{J,a}^2),\qquad e_J=\theta_J(1),$$
which bounds $\Cov_f^Q(Z_{i,J})$ in the positive semi-definite order whenever $\|f\|_\infty\leq U$ (Lemma~\ref{lem:release-covariance}), where its two terms account for sampling and for the privacy noise. With
$$v_J=\frac{\|\mathsf V_J\|_{\op}}{m_2},\qquad w_J=\Bigl\{\frac{2\operatorname{tr}(\mathsf V_J^2)}{m_2(m_2-1)}\Bigr\}^{1/2},$$
the conditional variance of $U_{m_2}(\widetilde f_{m_1},J)$ is at most $4v_J\|K_J(f-\widetilde f_{m_1})\|_2^2+w_J^2$, and $w_J$ is at most of order $2^{J/2}/m_2+(J\vee1)^{2\nu}2^{3J/2}/(m_2\alpha^2)$, the second term being due to privacy (Lemma~\ref{lem:finite-resolution-rates}).

Using that the variance of the U-statistic grows at most linearly with its expectation, it can then be turned into an upper confidence bound for the projected loss. There is an explicit increasing function $\mathcal H_{\gamma,J}$ such that, given the first split, $\|K_J(f-\widetilde f_{m_1})\|_2^2\leq\mathcal H_{\gamma,J}\{U_{m_2}(\widetilde f_{m_1},J)^+\}$ with probability at least $1-\gamma$ (Proposition~\ref{prop:projected-loss-upper}). We can then bound $\|f-\widetilde f_{m_1}\|_2^2\leq b_r(J) + \mathcal H_{\gamma,J}\{U_{m_2}(\widetilde f_{m_1},J)^+\}$ and select the resolution $J^\star$ balancing the two terms (Appendix~\ref{app:projected-loss-confidence}). The resolution $J^\star$ depends only on $m_2$, $\alpha$, $\gamma$ and the fixed parameters, so it is fixed before the data are seen.

With the centre $g_n=K_{J^\star}\widetilde f_{m_1}$ and the radius
\begin{equation}
\label{eq: Rn2}
R_n^2=b_r(J^\star)+\mathcal H_{\gamma,J^\star}\{U_{m_2}(g_n,J^\star)^+\},
\end{equation}
the confidence set is $C_n(\gamma)=\{h\in\Sigma(r,R,U):\|h-g_n\|_2^2\leq R_n^2\}$.

\begin{thm}
\label{thm:projected-loss}
Under the standing assumptions, $C_n(\gamma)$ is honest over $\Sigma(r,R,U)$ at level $1-\gamma$ for every $n$. For every $t\in[r,S)$, its diameter is of order $\ell_n\{\delta_{n,\est}(t,\alpha)\vee\rho_{n,\test}(r,\alpha)\}$ over $\Sigma(t,R,U)$.
\end{thm}

The proof is in Appendix~\ref{app:projected-loss-confidence}. Coverage holds for every $n$, without asymptotic approximation. The diameter combines the estimation rate of the centre with a floor $\ell_n\rho_{n,\test}(r,\alpha)$ due to the estimation of the projected loss, which is the testing rate of Section~\ref{sec:smoothness-testing}. Under bounded privacy, $\delta_{n,\est}(t,\alpha)\gtrsim\rho_{n,\test}(r,\alpha)$ exactly when $t\leq4r/3$. The single set $C_n(\gamma)$ therefore adapts simultaneously to every $t\in[r,4r/3]$, which proves the first statement of Theorem~\ref{thm:main-adaptation}. For $t>4r/3$, its diameter is of order $\ell_n\rho_{n,\test}(r,\alpha)$, which, for $\gamma<1/2$, no honest set can improve by more than logarithmic factors (Section~\ref{sec:smoothness-testing}).

\section{Beyond the boundary: adaptation under separation}
\label{sec:separated-adaptation}

When $\mathcal R$ contains some $s>4r/3$, the obstruction comes from rough densities whose distance to $\Sigma(s,R,U)$ is below the private testing rate. In the two-class problem formed by $r$ and $s$, following \citet{bullAdaptiveConfidenceSets2012}, we remove these densities and require honesty only over $\Sigma(s,R,U)\cup H_1(\rho_n)$, where $H_1(\rho)$ is the set from Section \ref{sec:smoothness-testing}.

The construction, detailed in Appendix~\ref{app:separated-adaptation}, uses three independent splits of sizes of order $n$. On the first, $\widetilde f_{m_1}$ from the previous section is projected in $L^2$ onto $\Sigma(s,R,U)$, giving a smooth reference $\overline f_s$ with risk at most $a_{m_1}(s)\lesssim\ell_n^2\delta_{n,\est}(s,\alpha)^2$ under the null (Proposition~\ref{prop:adaptive-private-estimator}). On the second, we test whether $f$ is smooth through the statistic $\widehat T=U_{m_2}(\overline f_s,J_{\sep})$ of Section~\ref{sec:construction}, which estimates the projected distance $\|K_{J_{\sep}}(f-\overline f_s)\|_2^2$, where the deterministic resolution $J_{\sep}$ results from a trade-off such that both Type-I and Type-II errors when testing $\widehat T>\tau$ are at most $\gamma_n=\min\{\gamma/2,\ell_n^{-2}\}$ (Appendix~\ref{app:separated-adaptation}), where 
\[\tau=\frac{1}{\gamma_n}\Bigl\{a_{m_1}(s)+\frac12\sqrt{4v_{J_{\sep}}a_{m_1}(s)+w_{J_{\sep}}^2}\Bigr\}.\] The estimated distance is small under the null, and large under $H_1(\rho_n)$, as $\overline f_s$ is smooth.

On the third split, we build the confidence set $C_n^{\den}(\gamma)=\{h\in\Sigma(t,R,U):\|h-\widetilde f_{m_3}\|_2^2\leq2a_{m_3}(t)/\gamma\}$, with $t=s$ if the above test accepts and $t=r$ if it rejects.

\begin{thm}
\label{thm:density-separated}
Assume the standing assumptions, let $r<s$, and let
$$\rho_n\geq M\ell_n^2\{\delta_{n,\est}(s,\alpha)\vee\rho_{n,\test}(r,\alpha)\}$$
for a sufficiently large constant $M$. The above confidence set $C_n^{\den}(\gamma)$ is honest over $\Sigma(s,R,U)\cup H_1(\rho_n)$ at level $1-\gamma$ for every $n$. Its diameter is of order $\ell_n\delta_{n,\est}(s,\alpha)$ over $\Sigma(s,R,U)$, and \[|C_n^{\den}(\gamma)|_2=O_{P_f^{Q^n}}\{\ell_n\delta_{n,\est}(r,\alpha)\}\] uniformly over $H_1(\rho_n)$.
\end{thm}

When $s>4r/3$, the required separation is of order $\ell_n^2\rho_{n,\test}(r,\alpha)$. By the argument of Section~\ref{sec:smoothness-testing}, no set that is honest over $\Sigma(s,R,U)\cup H_1(\rho_n)$ can have smooth-rate diameter when $\rho_n=o\{\rho_{n,\test}(r,\alpha)\}$, so the removed region cannot be made smaller, up to the factor $\ell_n^2$ (which we have not tried to optimise). We note that using the set of Section~\ref{sec:construction} after test rejection would add adaptation to every $t\in[r,4r/3]$ among separated densities, but we do not pursue this here.

\section{Numerical experiments}
\label{sec:numerics}

The theory of the preceding sections shows that local privacy affects the confidence set of Section~\ref{sec:construction} in two ways. First, the privacy noise is a second source of variability, besides sampling, in the statistic that estimates the projected loss. Second, this additional variability raises the price of honesty: over $\Sigma(t,R,U)$, the diameter is of order $\ell_n\{\delta_{n,\est}(t,\alpha)\vee\rho_{n,\test}(r,\alpha)\}$, where the floor $\rho_{n,\test}(r,\alpha)$ is set by the class $\Sigma(r,R,U)$ over which coverage is required, and under privacy this floor exceeds the estimation rate as soon as $t>4r/3$. We use simulations to examine how these effects appear at finite sample sizes. We focus on two questions. First, does accounting for the covariance introduced by the privacy mechanism matter for coverage? Second, how do the privacy level and the size of the honesty class affect the diameter of the resulting confidence set?

We study these questions using three fixed generating densities with different shapes: sine, triangular and normalised exponentiated Brownian (in decreasing order of regularity). To implement the continuous-density construction numerically, we approximate densities on a fine regular grid of $M=4096$ points on $[0,1]$ and use periodic Daubechies wavelets (with $S=8$) with a discrete Gram correction. Further implementation details, including the resulting finite-grid version of the Besov classes, are given in Appendix~\ref{app:numerical-implementation}. We set $R=1$, $U=3$, $\nu=1.25$ and $\gamma=0.05$, corresponding to nominal coverage $0.95$. Unless otherwise stated, the honesty class has regularity $r=0.5$ (matching the regularity of the Brownian dendsity), and all three generating densities belong to its finite-grid counterpart.

We consider sample sizes $n\in\{1000,5000,20000\}$ and $23$ privacy levels ranging from $\alpha=0.1$ to $5$, with larger $\alpha$ corresponding to weaker privacy. We also include the non-private construction, denoted by $\alpha=\infty$. Each configuration is evaluated over $5000$ Monte Carlo replications. Within a replication, procedures are compared using the same underlying samples, and private procedures share the same Laplace noise across values of $\alpha$. This reduces Monte Carlo variation in our comparisons, without changing the marginal distribution of any individual procedure.

\paragraph{Privacy and coverage.}
We first ask whether the additional uncertainty introduced by the privacy mechanism must indeed be accounted for to obtain reliable finite-sample coverage. In the construction of Section~\ref{sec:construction}, the covariance bound for the released wavelet coefficients contains both the ordinary sampling covariance and the covariance contributed by the added privacy noise. This bound affects two parts of the procedure: the choice of projection resolution and the upper confidence bound used to form the radius.

To assess the role of the privacy contribution, we conduct an ablation study using exactly the same privatised observations. The full procedure includes the privacy covariance in both the resolution selection and the confidence radius. The radius ablation selects the same resolution as the full procedure but omits the privacy covariance when constructing the radius. The full ablation omits the privacy covariance from both the resolution selection and the radius. 


\begin{figure*}[t]
\centering
\includegraphics[width=0.98\textwidth]{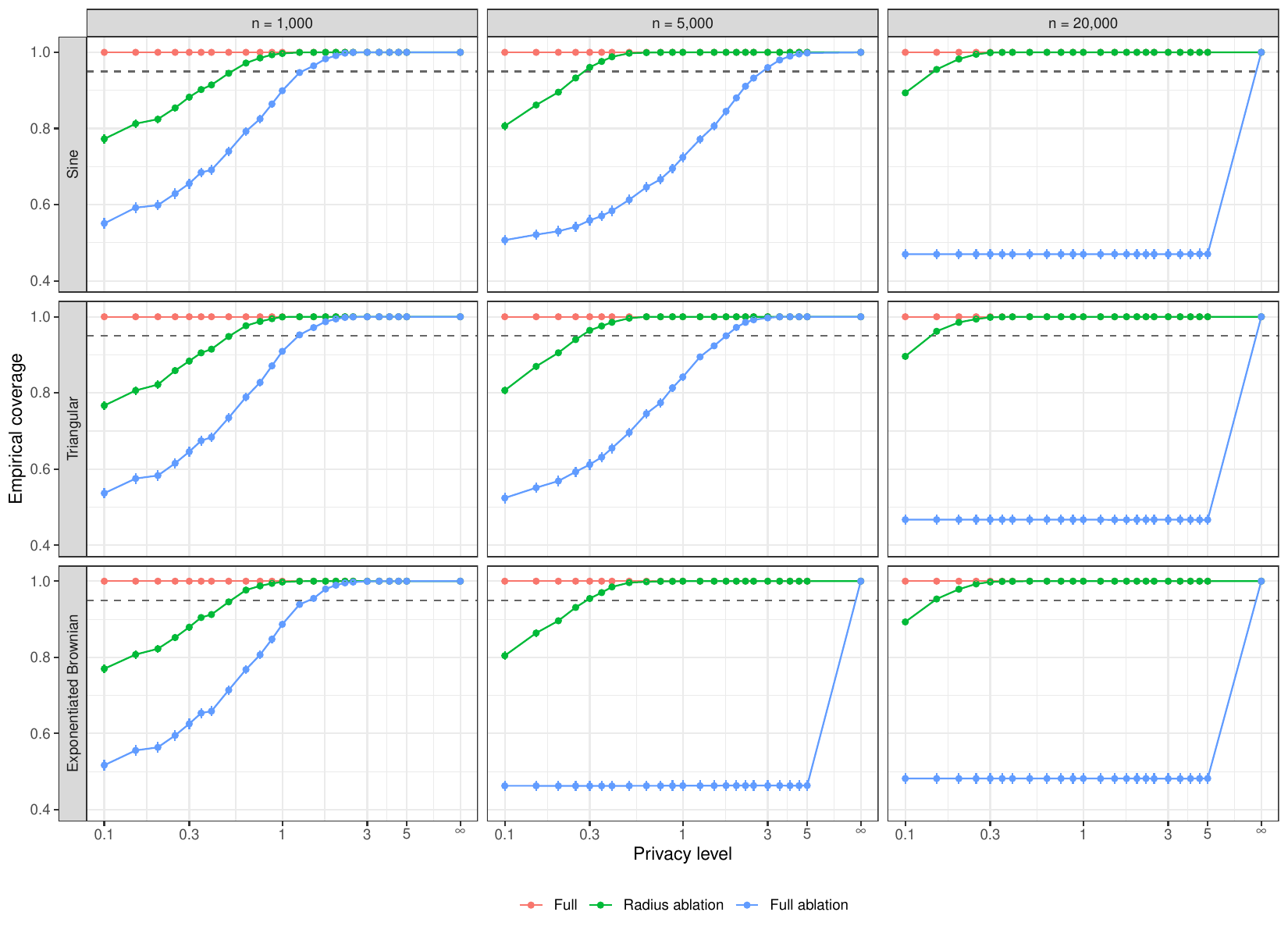}
\caption{\emph{Effect of accounting for privacy covariance on coverage.} Rows correspond to the sine, triangular and fixed exponentiated-Brownian densities, and columns to $n=1000$, $5000$ and $20000$. The full procedure accounts for privacy covariance in both resolution selection and the radius. The radius ablation omits it only from the radius, while the full ablation omits it from both. All procedures use the same privatised observations within each replication. Larger $\alpha$ corresponds to weaker privacy, and $\alpha=\infty$ denotes the common non-private endpoint. The dashed line marks the nominal coverage level $0.95$. Error bars are pointwise $95\%$ binomial confidence intervals based on $5000$ replications, nearly invisible due to their tightness.}
\label{fig:privacy-sweep}
\end{figure*}

Figure~\ref{fig:privacy-sweep} shows a clear distinction between the three settings. The full construction covers the generating density in every replication across the displayed configurations, being conservative and exceeding the nominal level $1-\gamma=0.95$. In contrast, coverage deteriorates when the privacy covariance is omitted, and the effect is stronger when it is completely removed from the construction. It falls as low as $0.767$ when the covariance is removed only from the radius and as low as $0.462$ when it is also removed from resolution selection. While the discrepancy diminishes with larger sample sizes and weaker privacy levels, it remains substantial even at the largest sample size and weakest finite privacy level considered: at $n=20000$ and $\alpha=5$, coverage of the full ablation ranges from $0.467$ to $0.482$ across the three densities. This behaviour is irrespective of the nominal regularity of the density. 

On the matter of conservatism of the full construction, this is expected from the way the finite-sample guarantee is obtained: the radius combines uniform covariance and approximation bounds with a one-sided Cantelli inequality, all chosen to guarantee coverage over the entire honesty class rather than to be sharp at any one fixed density. It results in conservative constant factors in the construction inflating the radius.



\paragraph{Privacy and the honesty class.}
Having established that privacy-induced uncertainty must be accounted for to maintain coverage, we next examine its effect on the size of the resulting confidence sets. Theorem~\ref{thm:projected-loss} predicts that the diameter depends not only on the privacy level and sample size, but also on the roughest class over which honesty is required. We therefore hold the sine generating density fixed and vary the regularity of the honesty class over $r\in\{0.2,0.3,0.4,0.5\}$. Since the Besov classes are nested, smaller $r$ corresponds to requiring coverage over a larger and less regular class. For each value of $r$, we recompute the deterministic projection resolution level $J^*$ from the corresponding design parameters and then recompute the confidence radius using the same observations and coupled privacy noise. We report $D=2R_n$ (as in \eqref{eq: Rn2}), the diameter of the enclosing ball. The case $r=0.5$ coincides with the baseline construction used in the previous paragraph.

\begin{figure*}[t]
\centering
\includegraphics[width=0.98\textwidth]{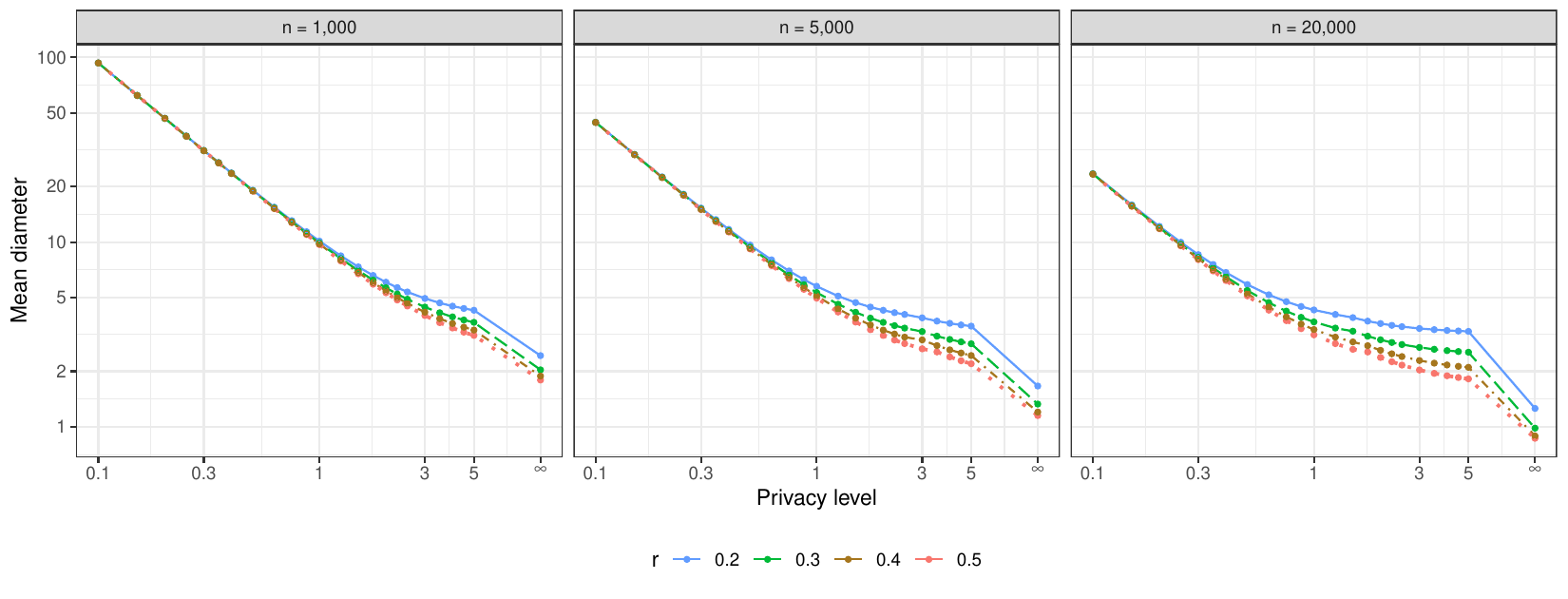}
\caption{\emph{Privacy and the cost of honesty for the sine density.} Panels correspond to $n=1000$, $5000$ and $20000$, and curves vary the regularity $r$ of the class over which coverage is required. Smaller $r$ corresponds to a larger honesty class. The generating sine density is held fixed throughout, so changes in diameter arise from changing the coverage requirement rather than the data-generating distribution. Points show the mean enclosing-ball diameter $D=2R_n$ over $5000$ replications, with error bars equal to $\pm1.96$ Monte Carlo standard errors (albeit, nearly invisible). Larger $\alpha$ corresponds to weaker privacy, and $\alpha=\infty$ denotes the non-private construction.}
\label{fig:honesty-cost}
\end{figure*}

Figure~\ref{fig:honesty-cost} clearly distinguishes between the two contributions to the size of the confidence set. First, the mean diameter decreases with increasing sample size and with weaker privacy, as the uncertainty in the private projected-loss estimate decreases. Second, even for the same generating density and observations, requiring honesty over a larger class increases the diameter. At $n=20000$ and $\alpha=5$, for example, decreasing $r$ from $0.5$ to $0.2$ increases the mean diameter from $1.819$ to $3.286$, an increase of $81\%$. This difference is much smaller under strong privacy: at $\alpha=0.1$, the corresponding increase is less than $1\%$. In this regime, privacy noise dominates the radius, leaving comparatively little visible effect from changing the honesty class. As privacy weakens, the dependence on the class over which coverage is required becomes increasingly apparent.

In Appendix \ref{app:numerical-implementation}, we additionally provide Table~\ref{tab:diameter-comparison}, which shows that the same behaviour is present for all three generating densities (sine, triangular, exponential-Brownian). Under strong privacy, their mean diameters are very similar because the privacy contribution dominates the projected loss in all cases. These differences become more visible as privacy weakens. At $n=20000$, $\alpha=1$ and $r=0.5$, for example, the mean diameters are $3.149$, $3.132$ and $3.095$ for the sine, triangular and exponentiated-Brownian densities, respectively, compared with non-private means of $0.844$, $0.800$ and $0.757$.

These results complement the coverage experiment above. Accounting for privacy-induced uncertainty is necessary to retain coverage, but the cost for doing so is a size-inflation of the confidence set. The size also depends on how large a class the procedure must cover honestly, even when the true density itself is unchanged. 


\section{Discussion}
\label{sec:discussion}

We have shown that, under bounded non-interactive local privacy, honest $L^2$-confidence sets over Besov classes can adapt to a range of regularities if and only if its largest element is at most $4/3$ times its smallest, against a factor $2$ without privacy. The boundary moves because local privacy degrades smoothness testing more than estimation, and honest confidence sets inherit the testing rate. We point out that our contribution is theoretical: the construction of Section~\ref{sec:construction} serves to show that the boundary is attained, and we do not claim that it is a practical procedure.

Adaptation under local privacy was known to come at a quantitative price, with adaptive estimators attaining the private minimax rates only up to logarithmic factors \citep{Butucea20, cai2025cost}, and our confidence sets carry a similar factor $\ell_n$, whose necessity we leave open. Our results show that the price is also qualitative: when $4r/3<s\leq2r$, adaptation is possible without privacy but impossible under local privacy, whatever the logarithmic factor allowed.

Our analysis is restricted to non-interactive channels, and extending it to sequentially interactive ones, in which each individual may use the reports of previous individuals, is a natural next step. Interaction is known to improve private goodness-of-fit testing \citep{BerrettButucea20} and the private estimation of quadratic functionals \citep{ButuceaQuadraticDP23}, on which smoothness testing relies, whereas it does not change the minimax rates of density estimation \citep{DJW18}. It is therefore conceivable that interaction enlarges the range over which honest adaptation is possible, a question that we leave for future work.


\bibliography{adapt_ci}

\newpage
\appendix

\section*{Notation}

For a density $f$ on $[0,1]$, $P_f$, $\E_f$ and $\Var_f$ refer to an observation $X\sim P_f$ before privatisation. When $m$ independent observations with density $f$ are passed through the channels $Q_i$, $P_f^{Q^m}$ and $\E_f^{Q^m}$ denote the joint law and the expectation of the reports $Z_1,\ldots,Z_m$, and $\E_f^{Q_i}$ the expectation of a function of the report $Z_i$ alone. We write $P_f^Q$, $\E_f^Q$, $\Var_f^Q$ and $\Cov_f^Q$ when the reports involved are clear from the context, as for the moments of a single report. For a probability measure $\nu$ on densities, $P_\nu^{Q^n}=\int P_f^{Q^n}\,d\nu(f)$. When a procedure uses independent sample splits, $P_2$, $\E_2$ and $\Var_2$ are conditional on the reports of the first split. For probability measures $\PP$ and $\PP'$, $\|\PP-\PP'\|_{\TV}=\sup_A|\PP(A)-\PP'(A)|$.

For functions, $\langle\cdot,\cdot\rangle$ and $\|\cdot\|_2$ are the inner product and norm of $L^2([0,1])$, and $\|\cdot\|_\infty$ is the supremum norm. For vectors, $\langle\cdot,\cdot\rangle$ is the Euclidean inner product and $\|\cdot\|_{\ell^p}$ the $\ell^p$ norm. For symmetric matrices, $A\preceq B$ means that $B-A$ is positive semidefinite, and $\|A\|_{\op}$ and $\|A\|_{\HS}=\{\operatorname{tr}(A^\top A)\}^{1/2}$ are the operator and Hilbert-Schmidt norms. We write $|A|$ for the cardinality of a finite set, $\1\{\cdot\}$ for indicators, $a\vee b=\max(a,b)$, $a\wedge b=\min(a,b)$ and $x^+=x\vee0$, and $\log$ is the natural logarithm. Throughout the appendices, $C$ denotes a positive constant that may change from line to line and depends only on the regularities involved, on $R$, $U$, $S$, $\nu$, $\bar\alpha$ and the wavelet basis, but never on $n$, $J$ or $\alpha$. $a\lesssim b$ means $a\leq Cb$, $a\asymp b$ means $a\lesssim b\lesssim a$, and a subscript, as in $\lesssim_\gamma$, allows $C$ to depend on the indicated quantity as well. Finally, $X_n=O_{P_f^{Q^n}}(r_n)$ uniformly over a class $\mathcal F$ means that $\lim_{M\to\infty}\limsup_{n\to\infty}\sup_{f\in\mathcal F}P_f^{Q^n}(|X_n|>Mr_n)=0$.

\section{Privacy mechanism and adaptive estimation}
\label{app:wavelets-privacy-estimation}

In this appendix, we construct the private adaptive estimator used as the centre of the confidence sets in Sections \ref{sec:construction} and \ref{sec:separated-adaptation}. To do this, we first recall the wavelet definition of the Besov classes introduced in Section \ref{sec:model} and describe the approximation bounds used throughout the paper. We then adapt the privacy mechanism from \citet{Butucea20}, defining locally private pointwise evaluation of wavelets, which we use to obtain our adaptive estimator. Focusing on the $B^t_{2,\infty}$ class of densities and the squared $L^2$-loss, we derive an explicit finite-sample risk bound that will be used in the later sections.


\subsection{Wavelet bases and Besov spaces}\label{app: bes_wav}
 
This section introduces the wavelet basis and the Besov norms used throughout the paper, and isolates the four properties of the basis on which all subsequent arguments rely.

For a fixed an integer $S_0>S$, let $\varphi$ and $\psi$ be the scaling function and the wavelet of a Daubechies multiresolution analysis of $L^2(\mathbb{R})$ with sufficiently many vanishing moments, such that $\varphi$ and $\psi$ are compactly supported and belong to $C^{S_0}(\mathbb{R})$, that $\int_{\mathbb{R}}\varphi=1$, and that $\psi$ is orthogonal to every polynomial of degree less than $S_0$. For every integer $j_0$, the functions $2^{j_0/2}\varphi(2^{j_0}\cdot-k)$, $k\in\mathbb{Z}$, and $2^{j/2}\psi(2^j\cdot-k)$, $j\geq j_0$, $k\in\mathbb{Z}$, then form an orthonormal basis of $L^2(\mathbb{R})$ (see \citet{daubechies1992ten}, \citet{meyer1992wavelets} and \citet[Chapter~4]{gineMathematicalFoundationsInfiniteDimensional2016}). Compact support matters below, because it bounds the number of basis functions that are non-zero at any given point, excluding for instance the Meyer basis.

Restricting our attention to the unit interval $[0,1]$, we will instead use the boundary-corrected basis constructed from $\varphi$ and $\psi$ by \citet{cohenWaveletsIntervalFast1993} (see also \citet[Section~4.3.5]{gineMathematicalFoundationsInfiniteDimensional2016}). For $j_0$ large enough for this construction to apply, the basis then consists of scaling functions $\varphi_{j_0k}$, $k\in N_{-1}$, at level $j_0$, and of wavelets $\psi_{jk}$, $k\in N_j$, at the levels $j\geq j_0$, where $N_{-1}$ and $N_j$ are finite index sets. There is a constant $C_\psi>0$, depending only on $\Phi$, $\Psi$ and $j_0$, such that the basis has the following properties.
\begin{enumerate}
\item[(W1)] The family $\{\varphi_{j_0k}:k\in N_{-1}\}\cup\{\psi_{jk}:j\geq j_0,\ k\in N_j\}$ is an orthonormal basis of $L^2([0,1])$.
\item[(W2)] $|N_{-1}|=2^{j_0}$, and $|N_j|=2^j$ for every $j\geq j_0$.
\item[(W3)]$\bigl\|\sum_{k\in N_{-1}}|\varphi_{j_0k}|\bigr\|_\infty\leq C_\psi2^{j_0/2}$ and, for every $j\geq j_0$, $\bigl\|\sum_{k\in N_j}|\psi_{jk}|\bigr\|_\infty\leq C_\psi2^{j/2}$.
\item[(W4)]  The constant function belongs to the span of $\{\varphi_{j_0k}:k\in N_{-1}\}$, or equivalently, $\int_0^1\psi_{jk}(x)\,dx=0$ for all $j\geq j_0$ and $k\in N_j$.
\end{enumerate}

Property (W3) is what controls the sensitivity of the privacy mechanism in Lemma~\ref{lem:release-privacy} and the positivity of the alternatives in Lemma~\ref{lem:probabilistic-positivity}, whereas (W4) yields the covariance bound of Lemma~\ref{lem:release-covariance} and ensures that the alternatives of Appendix~\ref{app:testing-lower-impossibility} integrate to one. We note that the periodised Daubechies basis of $L^2([0,1])$ \citep[Section~4.3.4]{gineMathematicalFoundationsInfiniteDimensional2016} has properties (W1) to (W4) as well, with $j_0=0$ and a single scaling function equal to one, and the implementation of Appendix~\ref{app:numerical-implementation} is based on a discretisation of it.
 
\paragraph{Besov spaces.}
For $f\in L^2([0,1])$, write $\alpha_{j_0k}(f)=\langle f,\varphi_{j_0k}\rangle$ for $k\in N_{-1}$, $\beta_{jk}(f)=\langle f,\psi_{jk}\rangle$ for $j\geq j_0$ and $k\in N_j$, and $\beta_j(f)=(\beta_{jk}(f))_{k\in N_j}$. For $0<t<S_0$, the norm used in Section~\ref{sec:model} is
\begin{align*}
\|f\|_{B^t_{2,\infty}}
\coloneqq\max\Bigl\{&\Bigl(\sum_{k\in N_{-1}}\alpha_{j_0k}(f)^2\Bigr)^{1/2},\\
&\sup_{j\geq j_0}2^{jt}\|\beta_j(f)\|_{\ell^2}\Bigr\},
\end{align*}
and we let $B^t_{2,\infty}([0,1])$ denote the space of those $f$ for which this quantity is finite. For $0<t<S_0$, it is equivalent to the usual norm of the Besov space $B^t_{2,\infty}([0,1])$, and another sufficiently regular boundary-corrected basis would produce an equivalent norm \citep[Section~4.3.5]{gineMathematicalFoundationsInfiniteDimensional2016}. Since, by (W1) and (W4), the constant function is a unit vector in the span of the scaling functions, every density $f$ satisfies $(\sum_{k\in N_{-1}}\alpha_{j_0k}(f)^2)^{1/2}\geq|\langle f,1\rangle|=1$ by Cauchy-Schwarz inequality, with equality for the uniform density. It is for this reason that Section~\ref{sec:model} requires $R\geq1$, and why the uniform density belongs to every class $\Sigma(t,R,U)$.

For $j\geq j_0$, let $E_j(f)=\|\beta_j(f)\|_{\ell^2}^2$ denote the energy of $f$ at level $j$. For $J>j_0$, let $K_J$ be the orthogonal projection onto the span of the scaling functions and of the wavelets at levels $j_0,\ldots,J-1$. We index these functions jointly by
$$
\begin{aligned}
\calI_J
&=\{(-1,k):k\in N_{-1}\}\\
&\quad\cup\{(j,k):j_0\leq j\leq J-1,\ k\in N_j\},
\end{aligned}
$$
writing $\psi_a$ for the scaling function or wavelet with index $a\in\calI_J$ and $j(a)$ for its level, with the convention $j(a)=-1$ for the scaling functions, to which the notation $N_{-1}$ refers. Then
$$K_Jf=\sum_{a\in\calI_J}\langle f,\psi_a\rangle\psi_a,$$
and, by (W2), $|\calI_J|=2^{j_0}+\sum_{j=j_0}^{J-1}2^j=2^J$. This indexing is used for the private release below and for the projected-loss statistic of Appendix~\ref{app:projected-loss-confidence}. The next two lemmas record the consequences of (W1) to (W3) that are used repeatedly.
 
\begin{lem}
\label{lem:wavelet-basic-bounds}
For every $j\geq j_0$, one has $\|\psi_{jk}\|_\infty\leq C_\psi2^{j/2}$ for all $k\in N_j$, and, for every coefficient vector $b_j=(b_{jk})_{k\in N_j}$,
$$\Bigl\|\sum_{k\in N_j}b_{jk}\psi_{jk}\Bigr\|_\infty\leq C_\psi2^{j/2}\max_{k\in N_j}|b_{jk}|\leq C_\psi2^{j/2}\|b_j\|_{\ell^2}.$$
Similarly, $\|\varphi_{j_0k}\|_\infty\leq C_\psi2^{j_0/2}$ for all $k\in N_{-1}$.
\end{lem}
 
\begin{proof}
For $x\in[0,1]$, (W3) gives
$$
\begin{aligned}
\Bigl|\sum_{k\in N_j}b_{jk}\psi_{jk}(x)\Bigr|
&\leq\max_{k\in N_j}|b_{jk}|\sum_{k\in N_j}|\psi_{jk}(x)|\\
&\leq C_\psi2^{j/2}\max_{k\in N_j}|b_{jk}|,
\end{aligned}
$$
and $\max_{k\in N_j}|b_{jk}|\leq\|b_j\|_{\ell^2}$. The bounds on $\|\psi_{jk}\|_\infty$ and $\|\varphi_{j_0k}\|_\infty$ follow from (W3), since each term of a sum of non-negative functions is bounded by the sum.
\end{proof}
 
\begin{lem}
\label{lem:besov-tail}
Let $0<t<S$ and $\|f\|_{B^t_{2,\infty}}\leq R$. Then $E_j(f)\leq R^22^{-2jt}$ for every $j\geq j_0$ and, for every $J>j_0$,
\[\|(I-K_J)f\|_2^2\leq b_t(J)\coloneqq\frac{R^22^{-2Jt}}{1-2^{-2t}}.\]
\end{lem}
 
\begin{proof}
The first bound restates $2^{jt}\|\beta_j(f)\|_{\ell^2}\leq R$. Since $(I-K_J)f$ consists of the levels $j\geq J$, (W1) and Parseval's identity give
$$
\begin{aligned}
\|(I-K_J)f\|_2^2&=\sum_{j=J}^\infty E_j(f)\leq R^2\sum_{j=J}^\infty2^{-2jt}\\
&=\frac{R^22^{-2Jt}}{1-2^{-2t}}.
\end{aligned}
$$
\end{proof}

\subsection{Level-dependent private wavelet release}\label{sec: privacy release}

We now define the private release used both for the adaptive estimator and for the projected-loss statistic. It is adapted from the privacy mechanism of \citet{Butucea20}, applied to our boundary-corrected basis: we release point-evaluation of wavelets, with those at level $j$ being perturbed by independent Laplace noise with scale proportional to $(j\vee1)^\nu2^{j/2}/\alpha$, where $\nu>1$. The factor $2^{j/2}$ accounts for the size of the wavelets at level $j$, while $(j\vee1)^\nu$ distributes the privacy budget across levels so that the total sensitivity remains bounded. \citet{Butucea20} derive the constant of proportionality from the bounded supports of the translates of their wavelets on the real line. In order to extend the analysis to the modified edge functions of the boundary-corrected basis, we take it equal to the exact weighted sensitivity $\Delta_J$ defined below, which Lemma~\ref{lem:release-privacy} shows to be bounded uniformly in $J$.

We first fix $J>j_0$, define
\[d_{-1}\coloneqq2^{j_0/2},\qquad d_{j}\coloneqq(j\vee1)^\nu2^{j/2},\qquad j\geq j_0,\]
and write $d_{j(a)}$ for the corresponding quantity when $a\in\calI_J$. Define the weighted sensitivity
\begin{equation}\label{eq: delta_j}\Delta_J\coloneqq\sup_{x,x'\in[0,1]}\sum_{a\in\calI_J}\frac{|\psi_a(x)-\psi_a(x')|}{d_{j(a)}},\end{equation}
and set
$$\sigma_{J,a}=\frac{\Delta_Jd_{j}}{\alpha}.$$
The $i$th unit corresponding to $X_i$ releases the privatized vector $Z_{i,J}=(Z_{i,J,a})_{a\in\calI_J}$ with
$$Z_{i,J,a}=\psi_a(X_i)+\sigma_{J,a}W_{i,a},\qquad a\in\calI_J,$$
where the $W_{i,a}$ are independent standard Laplace variables with density $w\mapsto e^{-|w|}/2$. The resolution $J$ is fixed before observing the data and the channel is applied similarly to all.

The next lemma verifies that this release satisfies the privacy requirement from Section \ref{sec:model}. We will also have use of the uniform bound on $\Delta_J$ as It shows that releasing coefficients at increasingly fine resolutions does not introduce any additional dependence on $J$, beyond the prescribed weights $d_{j(a)}$.

\begin{lem}
\label{lem:release-privacy}
For every deterministic $J>j_0$, the release $Z_{i,J}$ is non-interactive $\alpha$-LDP. Moreover,
$$\sup_{J>j_0}\Delta_J<\infty.$$
\end{lem}

\begin{proof}
Fix $J>j_0$. Conditional on $X_i=x$, independence of the Laplace variables implies that $Z_{i,J}$ has density
$$q_J(z\mid x)=\prod_{a\in\calI_J}\frac{1}{2\sigma_{J,a}}\exp\left\{-\frac{|z_a-\psi_a(x)|}{\sigma_{J,a}}\right\}.$$
Hence, for $x,x'\in[0,1]$,
$$\log\frac{q_J(z\mid x)}{q_J(z\mid x')}=\sum_{a\in\calI_J}\frac{|z_a-\psi_a(x')|-|z_a-\psi_a(x)|}{\sigma_{J,a}}.$$
By the reverse triangle inequality,
$$|z_a-\psi_a(x')|-|z_a-\psi_a(x)|\leq|\psi_a(x)-\psi_a(x')|.$$
Using the definition of $\sigma_{J,a}$ therefore gives
$$\log\frac{q_J(z\mid x)}{q_J(z\mid x')}\leq\frac{\alpha}{\Delta_J}\sum_{a\in\calI_J}\frac{|\psi_a(x)-\psi_a(x')|}{d_{j(a)}}\leq\alpha.$$
Thus $q_J(z\mid x)\leq e^\alpha q_J(z\mid x')$ for every $z,x,x'$. Integrating this inequality over any measurable set of outputs proves the $\alpha$-LDP condition. Since the release of individual $i$ depends only on $X_i$ and independent randomness, the mechanism is non-interactive.

It remains to show that $\Delta_J$ is bounded uniformly in $J$. By property (W3) of Appendix~\ref{app: bes_wav}, for every $x\in[0,1]$,
$$\sum_{k\in N_{-1}}\frac{|\varphi_{j_0k}(x)|}{2^{j_0/2}}\leq C_\psi,$$
and, for every resolution level $j\geq j_0$,
$$\sum_{k\in N_j}\frac{|\psi_{jk}(x)|}{(j\vee1)^\nu2^{j/2}}\leq C_\psi(j\vee1)^{-\nu}.$$
Applying these bounds at both $x$ and $x'$ yields
$$\Delta_J\leq2C_\psi\left\{1+\sum_{j=j_0}^{J-1}(j\vee1)^{-\nu}\right\}.$$
Since $\nu>1$, the series $\sum_{j\geq j_0}(j\vee1)^{-\nu}$ converges. Its partial sums are therefore uniformly bounded in $J$, which proves $\sup_{J>j_0}\Delta_J<\infty$.
\end{proof}

The adaptive estimator in Appendix \ref{app: adaptive_estimator} averages these released coefficients over the observations. We then characterize their expectations, variances, and the order of the added privacy noise. The releases are unbiased statistics since it contains a zero-mean Laplace noise. As well, the variance can be separate into the sampling contribution and the independent Laplace contribution. The uniform sensitivity bound from Lemma \ref{lem:release-privacy} then controls the latter.

\begin{lem}
\label{lem:release-moments}
Let $f$ be a density satisfying $\|f\|_\infty\leq U$ for $U>0$. Then, for every $J>j_0$ and $a\in\calI_J$,
\[\E_f^Q Z_{i,J,a}=\langle f,\psi_a\rangle,\qquad \Var_f^Q(Z_{i,J,a})\leq U+2\sigma_{J,a}^2.\]
Moreover, uniformly over $J>j_0$ and $a\in\calI_J$,
$$\sigma_{J,a}^2\leq C\alpha^{-2}d_{j(a)}^2.$$
\end{lem}

\begin{proof}
Since $W_{i,a}$ is independent of $X_i$, centred and of variance two, $\E_f^QZ_{i,J,a}=\E_f\{\psi_a(X_i)\}=\langle f,\psi_a\rangle$ and
$$
\begin{aligned}
\Var_f^Q(Z_{i,J,a})&=\Var_f\{\psi_a(X_i)\}+2\sigma_{J,a}^2\\
&\leq\int_0^1\psi_a(x)^2f(x)\,dx+2\sigma_{J,a}^2\\
&\leq U+2\sigma_{J,a}^2,
\end{aligned}
$$
where the last step uses $f\leq U$ and $\|\psi_a\|_2=1$, by (W1). Finally, $\sigma_{J,a}^2=\alpha^{-2}\Delta_J^2d_{j(a)}^2\leq C\alpha^{-2}d_{j(a)}^2$ uniformly over $J$ and $a$, since $\sup_{J>j_0}\Delta_J<\infty$ by Lemma~\ref{lem:release-privacy}.
\end{proof}

\subsection{Adaptive private estimator}\label{app: adaptive_estimator}

We now construct the private adaptive estimator used in Sections \ref{sec:construction} and \ref{sec:separated-adaptation}, based on wavelet thresholding as in \citet{Butucea20}. Their result applies to general Besov classes $B^s_{p,q}$ under $L^r$-loss, but we focus here on the $L^2$-loss over $B^t_{2,\infty}$. This allows a simpler argument based on simultaneous concentration of the released coefficients, which will be crucial for the construction of the confidence set in Appendix \ref{app:projected-loss-confidence}.

The estimator is computed from the private reports of a single split of the sample, of size $m_1$ in Section~\ref{sec:construction} and of sizes $m_1$ and $m_3$ in Section~\ref{sec:separated-adaptation}. We therefore describe it for a generic split of size $m\geq1$. Its construction does not involve the unknown smoothness $t$.

By Lemma~\ref{lem:release-moments}, a coefficient released at level $j$ has variance at most of order $d_{j}^2/\alpha^2$, so that its average over the split has variance at most of order $d_{j(a)}^2/(m\alpha^2)$ and the effective sample size is therefore
\[N_m=m\alpha^2.\]
It determines the finest resolution from which signal can be recovered: we gather from observations in the split their coefficients through level $J=K_m-1$, where
\[K_m\coloneqq\max\Bigl\{j_0+1,\Bigl\lceil\tfrac12\log_2(e+N_m)\Bigr\rceil\Bigr\},\]
so that $2^{K_m}\asymp N_m^{1/2}$ for large $N_m$, and the levels left out contribute at most $b_t(K_m)\lesssim N_m^{-t}$ to the squared error (Lemma~\ref{lem:besov-tail}), which is of smaller order than the estimation rate for every $t>0$. For $a\in\calI_{K_m}$, write $\theta_a=\langle f,\psi_a\rangle$. By Lemma~\ref{lem:release-moments},
$$\widehat\theta_a=\frac1m\sum_{i=1}^mZ_{i,K_m,a}$$
is an unbiased estimator of $\theta_a$.

Write $\sigma_a=\sigma_{K_m,a}=\Delta_{K_m}d_{j(a)}/\alpha$ for the noise scale of coefficient $a$,  $B_a=\|\psi_a\|_\infty$ and $c_a=\max\{2B_a/3,\sigma_a\}$, and assuming $\|f_0\|_\infty\leq U$ for some known $U>0$, set the threshold $\tau_a$ as
$$\varepsilon_m=(e+N_m)^{-2},\qquad x_m=\log(2^{K_m+1}/\varepsilon_m),$$
$$\tau_a=2\left\{\sqrt{\frac{2(U+2\sigma_a^2)x_m}{m}}+\frac{c_ax_m}{m}\right\}.$$
The proof of Proposition~\ref{prop:adaptive-private-estimator} below shows that $|\widehat\theta_a-\theta_a|\leq\tau_a/2$ for all $2^{K_m}$ released coefficients index by $\calI_{K_m}$ simultaneously, with probability at least $1-\varepsilon_m$.

Given the thresholded coefficients $\widetilde\theta_a=\widehat\theta_a\1\{|\widehat\theta_a|>\tau_a\}$, we define
\[\widehat f_m\coloneqq\sum_{\substack{a\in\calI_{K_m}\\j(a)=-1}}\widehat\theta_a\psi_a+\sum_{\substack{a\in\calI_{K_m}\\j(a)\geq j_0}}\widetilde\theta_a\psi_a,\]
so that our estimator $\widetilde f_m$ is the $L^2$ projection of $\widehat f_m$ onto the ball $\{h:\|h\|_2\leq\sqrt U\}$. This ball contains $\Sigma(t,R,U)$, because $\|f\|_2^2\leq\|f\|_\infty\int_0^1f\leq U$, so the projection can only reduce the approximation error, and it bounds the squared error by $4U$ on the event where the concentration above fails.

For $0<t<S$, let $a_m(t)\coloneqq \min\{4U,A_m(t)\}$, where
$$
\begin{aligned}
A_m(t)&=\sum_{\substack{a\in\calI_{K_m}\\j(a)=-1}}\frac{U+2\sigma_a^2}{m}+b_t(K_m)+4U\varepsilon_m\\
&\quad+\frac94\sum_{j=j_0}^{K_m-1}\min\Bigl\{R^22^{-2jt},\sum_{a:j(a)=j}\tau_a^2\Bigr\}.
\end{aligned}
$$
The four terms account for the coefficients at the smallest resolution level, the levels that are discarded, the event on which the concentration fails, and the thresholded coefficients. Finally, $\ell_m=\{\log(e+m\alpha^2)\}^{\nu+1/2}$ and $\delta_{m,\est}(t,\alpha)=m^{-t/(2t+1)}\vee(m\alpha^2)^{-t/(2t+2)}$ are the quantities $\ell_n$ and $\delta_{n,\est}(t,\alpha)$ of Section~\ref{sec:model} with $n$ replaced by $m$.

\begin{prop}
\label{prop:adaptive-private-estimator}
For every $m\geq1$ and $0<t<S$,
$$\sup_{f\in\Sigma(t,R,U)}\E_f^{Q^m}\|\widetilde f_m-f\|_2^2\leq a_m(t).$$
Under bounded privacy $\alpha\leq \bar \alpha$,
$$a_m(t)\lesssim\ell_m^2\delta_{m,\est}(t,\alpha)^2.$$
\end{prop}

\begin{proof}
We prove the two assertions in turn, starting with the risk bound in three steps. We first show that, on an event $\mathcal E_m$ of probability at least $1-\varepsilon_m$, every released coefficient lies within $\tau_a/2$ of its mean. We then bound the error of each thresholded coefficient on $\mathcal E_m$. Finally, we combine these bounds through Parseval's identity, together with the contribution of the levels that are not released and a crude bound off $\mathcal E_m$.

Fix $a\in\calI_{K_m}$. The coefficient released by individual $i$ decomposes as $Z_{i,K_m,a}-\theta_a=Y_{i,a}+\sigma_aW_{i,a}$, where $Y_{i,a}=\psi_a(X_i)-\theta_a$ is the sampling error and the Laplace variable $W_{i,a}$ is independent of $Y_{i,a}$. The sampling error is centred, satisfies $|Y_{i,a}|\leq2B_a$ because $|\psi_a(X_i)|\leq B_a$ and $|\theta_a|\leq\E_f|\psi_a(X_i)|\leq B_a$, and has variance at most $U$, as in the proof of Lemma~\ref{lem:release-moments}. Bernstein's bound on its moment generating function therefore gives, for $|\lambda|<3/(2B_a)$,
$$\log\E_fe^{\lambda Y_{i,a}}\leq\frac{U\lambda^2}{2(1-2B_a|\lambda|/3)}.$$
For the privacy noise, $\E e^{uW_{i,a}}=(1-u^2)^{-1}$ for $|u|<1$, and since
$$-\log(1-u^2)=\sum_{k=1}^\infty\frac{u^{2k}}{k}\leq\frac{u^2}{1-u^2}\leq\frac{u^2}{1-|u|},$$
we get for $|\lambda|<\sigma_a^{-1}$,
$$\log\E e^{\lambda\sigma_aW_{i,a}}\leq\frac{\sigma_a^2\lambda^2}{1-\sigma_a|\lambda|}=\frac{2\sigma_a^2\lambda^2}{2(1-\sigma_a|\lambda|)}.$$
By independence the two bounds add, and since $2B_a/3\leq c_a$ and $\sigma_a\leq c_a$, for $|\lambda|<c_a^{-1}$,
$$\log\E_f^Qe^{\lambda(Z_{i,K_m,a}-\theta_a)}\leq\frac{(U+2\sigma_a^2)\lambda^2}{2(1-c_a|\lambda|)}.$$
Since $\widehat\theta_a-\theta_a$ is the average of $m$ independent copies of $Z_{i,K_m,a}-\theta_a$, it follows that, with $\kappa_a=(U+2\sigma_a^2)/m$ and for $|\lambda|<m/c_a$,
$$\log\E_f^{Q^m}e^{\lambda(\widehat\theta_a-\theta_a)}\leq\frac{\kappa_a\lambda^2}{2(1-c_a|\lambda|/m)}.$$
Fix $x>0$ and put $\zeta_a(x)=\sqrt{2\kappa_ax}+c_ax/m$. For $\lambda=\sqrt{2\kappa_ax}/\{\kappa_a+(c_a/m)\sqrt{2\kappa_ax}\}$, which lies in $(0,m/c_a)$, a direct computation gives
$$\lambda\zeta_a(x)-\frac{\kappa_a\lambda^2}{2(1-c_a\lambda/m)}=x,$$
so that Markov's inequality applied to $e^{\lambda(\widehat\theta_a-\theta_a)}$ yields $P_f^{Q^m}\{\widehat\theta_a-\theta_a>\zeta_a(x)\}\leq e^{-x}$. The bound on the logarithmic moment generating function depends on $\lambda$ only through $|\lambda|$, so the lower tail is handled in the same way, and
$$P_f^{Q^m}\{|\widehat\theta_a-\theta_a|>\zeta_a(x)\}\leq2e^{-x}.$$
By definition of $\tau_a$, $\zeta_a(x_m)=\tau_a/2$. Let
$$\mathcal E_m=\bigcap_{a\in\calI_{K_m}}\{|\widehat\theta_a-\theta_a|\leq\tau_a/2\}.$$
A union bound over the $|\calI_{K_m}|=2^{K_m}$ released coefficients gives
$$P_f^{Q^m}(\mathcal E_m^c)\leq2^{K_m+1}e^{-x_m}=\varepsilon_m.$$

We now bound the error of each thresholded coefficient on $\mathcal E_m$, where $|\widehat\theta_a-\theta_a|\leq\tau_a/2$ for every released $a$. Let $a$ be a detail coefficient. If it is discarded, then $\widetilde\theta_a=0$ and $|\widehat\theta_a|\leq\tau_a$. The error $(\widetilde\theta_a-\theta_a)^2$ equals $\theta_a^2$, and it is at most $9\tau_a^2/4$ because $|\theta_a|\leq|\widehat\theta_a|+\tau_a/2\leq3\tau_a/2$. If it is retained, then $\widetilde\theta_a=\widehat\theta_a$ and $|\widehat\theta_a|>\tau_a$. The error is at most $\tau_a^2/4$, and it is at most $\theta_a^2$ because $|\theta_a|\geq|\widehat\theta_a|-\tau_a/2>\tau_a/2$. In both cases the error is bounded by $\theta_a^2$ and by $9\tau_a^2/4$, so that, on $\mathcal E_m$,
$$(\widetilde\theta_a-\theta_a)^2\leq\frac94\min\{\theta_a^2,\tau_a^2\}.$$

Summing over a fixed level $j$ gives
\begin{align*}
\sum_{a:j(a)=j}(\widetilde\theta_a-\theta_a)^2&\leq\frac94\sum_{a:j(a)=j}\min\{\theta_a^2,\tau_a^2\}\\
&\leq\frac94\min\left\{\sum_{a:j(a)=j}\theta_a^2,\sum_{a:j(a)=j}\tau_a^2\right\}.
\end{align*}

We finally combine these bounds. By (W1), the error of $\widehat f_m$ splits into the contributions of the different coefficients:
$$
\begin{aligned}
\|\widehat f_m-f\|_2^2
&=\sum_{\substack{a\in\calI_{K_m}\\j(a)=-1}}(\widehat\theta_a-\theta_a)^2\\
&\quad+\sum_{j=j_0}^{K_m-1}\sum_{a:j(a)=j}(\widetilde\theta_a-\theta_a)^2
+\sum_{j=K_m}^\infty E_j(f).
\end{aligned}
$$
The last sum is at most $b_t(K_m)$ by Lemma~\ref{lem:besov-tail}, and the middle one is controlled on $\mathcal E_m$ by the previous display. The scaling coefficients are not thresholded, and by independence across individuals and Lemma~\ref{lem:release-moments}, $\E_f^{Q^m}(\widehat\theta_a-\theta_a)^2=\Var_f^Q(Z_{i,K_m,a})/m\leq(U+2\sigma_a^2)/m$ for each of them. It remains to account for the projection and for $\mathcal E_m^c$. Since $\|f\|_2^2\leq\|f\|_\infty\int_0^1f\leq U$, the density $f$ lies in the closed convex ball onto which $\widehat f_m$ is projected, and projection onto a closed convex set does not increase distances to its points, so that $\|\widetilde f_m-f\|_2\leq\|\widehat f_m-f\|_2$. As $\widetilde f_m$ and $f$ both lie in that ball, also $\|\widetilde f_m-f\|_2^2\leq4U$. Therefore
$$\|\widetilde f_m-f\|_2^2\leq\|\widehat f_m-f\|_2^2\1_{\mathcal E_m}+4U\1_{\mathcal E_m^c},$$
and taking expectations, with the three bounds above and $P_f^{Q^m}(\mathcal E_m^c)\leq\varepsilon_m$, gives $\E_f^{Q^m}\|\widetilde f_m-f\|_2^2\leq A_m(t)$. Together with $\|\widetilde f_m-f\|_2^2\leq4U$, this yields $\E_f^{Q^m}\|\widetilde f_m-f\|_2^2\leq a_m(t)$ for every $f\in\Sigma(t,R,U)$, which proves the risk bound.

It remains to establish the rate under bounded privacy. In this part, constants may also depend on $t$. Put $H_m=\log(e+N_m)$, so that $\ell_m^2=H_m^{2\nu+1}$. Since $a_m(t)\leq4U$ and $\ell_m^2\delta_{m,\est}(t,\alpha)^2\geq N_m^{-t/(t+1)}$, the claim is immediate when $N_m$ is bounded, and we may then assume that $N_m$ exceeds a constant. Then $K_m=\lceil\frac12\log_2(e+N_m)\rceil$, so that
$$K_m\lesssim H_m,\qquad 2^{K_m}\asymp N_m^{1/2},\qquad x_m\lesssim H_m.$$
We bound the four terms of $A_m(t)$ in turn. There are $2^{j_0}$ coefficients at the initial resolution level, and $U+2\sigma_a^2\lesssim\alpha^{-2}$ for each of them by Lemma~\ref{lem:release-moments} and $\alpha\leq\bar\alpha$, so the first term is $O(N_m^{-1})$. Next, $b_t(K_m)\lesssim2^{-2K_mt}\leq N_m^{-t}$ and $4U\varepsilon_m\lesssim N_m^{-2}$. For the other coefficients, let $a$ be an index at level $j<K_m$. Lemmas~\ref{lem:wavelet-basic-bounds} and~\ref{lem:release-moments}, with $\alpha\leq\bar\alpha$, give $B_a^2\lesssim2^j\lesssim\alpha^{-2}2^j$ and $U+2\sigma_a^2\lesssim\alpha^{-2}(j\vee1)^{2\nu}2^j$, hence also $c_a^2\lesssim\alpha^{-2}(j\vee1)^{2\nu}2^j$. Since $(x+y)^2\leq2x^2+2y^2$, $x_m\lesssim H_m$ and $H_m\lesssim m$,
$$\tau_a^2\lesssim\frac{(j\vee1)^{2\nu}2^j}{m\alpha^2}\Bigl(x_m+\frac{x_m^2}{m}\Bigr)\lesssim\frac{(j\vee1)^{2\nu}2^jH_m}{N_m}.$$
As level $j$ contains $2^j$ coefficients and $j\vee1\leq K_m\lesssim H_m$,
$$\sum_{a:j(a)=j}\tau_a^2\lesssim\frac{2^{2j}}{R_m},\qquad R_m=\frac{N_m}{H_m^{2\nu+1}}.$$
Let $k_m$ be an integer with $2^{k_m}\asymp R_m^{1/(2t+2)}$. Since $R_m\to\infty$ and $R_m^{1/(2t+2)}\leq N_m^{1/(2t+2)}$ is of smaller order than $2^{K_m}\asymp N_m^{1/2}$, one has $j_0\leq k_m<K_m$ for $N_m$ large enough. Using the bound in $\tau_a$ at the levels $j\leq k_m$ and the bound $R^22^{-2jt}$ at the levels $j>k_m$, we obtain
$$
\begin{aligned}
&\sum_{j=j_0}^{K_m-1}\min\Bigl\{R^22^{-2jt},\sum_{a:j(a)=j}\tau_a^2\Bigr\}\\
&\quad\lesssim\frac1{R_m}\sum_{j=j_0}^{k_m}2^{2j}+\sum_{j>k_m}2^{-2jt}\\
&\quad\lesssim\frac{2^{2k_m}}{R_m}+2^{-2k_mt}\asymp R_m^{-t/(t+1)},
\end{aligned}
$$
the choice of $k_m$ balancing the two terms. Collecting the four bounds,
$$
\begin{aligned}
a_m(t)&\lesssim R_m^{-t/(t+1)}+N_m^{-1}+N_m^{-t}+N_m^{-2}\\
&\lesssim H_m^{(2\nu+1)t/(t+1)}N_m^{-t/(t+1)},
\end{aligned}
$$
because $t/(t+1)<\min\{1,t\}$ and $H_m\geq1$. Finally, $H_m^{(2\nu+1)t/(t+1)}\leq H_m^{2\nu+1}=\ell_m^2$ and $N_m^{-t/(t+1)}\leq\delta_{m,\est}(t,\alpha)^2$, so that $a_m(t)\lesssim\ell_m^2\delta_{m,\est}(t,\alpha)^2$, which concludes the proof.
\end{proof}

\section{Confidence set based on projected loss}
\label{app:projected-loss-confidence}

This appendix proves Theorem~\ref{thm:projected-loss}, and with it the positive side of Theorem~\ref{thm:main-adaptation}. The construction follows the projected-loss approach to $L^2$ confidence sets of \citet{robinsAdaptiveNonparametricConfidence2006} and \citet{bullAdaptiveConfidenceSets2012}, presented in \citet{gineMathematicalFoundationsInfiniteDimensional2016}: an independent sample split is used to estimate the squared projected error of a preliminary estimator by a quadratic statistic, and the radius is derived from this estimate. \citet{lam-weilMinimaxOptimalGoodnessoffit2021} use a closely related statistic of locally private coefficients for goodness-of-fit testing. In our setting, the centre is the adaptive private estimator of Appendix~\ref{app:wavelets-privacy-estimation}, computed on the first split, and the statistic is formed from the private wavelet coefficients released on the second split. The main difference with the non-private construction is that the variance of the statistic, and hence the radius, must account for the covariance of the privacy noise.

Throughout this appendix, $m\geq2$ denotes the size of the second split, on which the loss is estimated, and the resolution $J>j_0$ is fixed independently of the data of that split. The quantities $v_J$, $w_J$, $\mathcal H_{\eta,J}$ and $F_\eta(J)$ are those of Section~\ref{sec:construction}, with $m$ in place of $m_2$.

\subsection{Conditional moment bounds}

The projected-loss statistic is built from inner products of the released coefficient vectors, so its variance is governed by the covariance of a single private release. We first bound this covariance by a deterministic matrix, uniformly over densities bounded by $U$. Recall from Section~\ref{sec:construction} that $e_J=\theta_J(1)=(\langle1,\psi_a\rangle)_{a\in\calI_J}$ and
$$\mathsf V_J=U(I-e_Je_J^\top)+2\operatorname{diag}_{a\in\calI_J}(\sigma_{J,a}^2).$$
The first term bounds the covariance of the sampling part $(\psi_a(X_i))_{a\in\calI_J}$ of the release; the projection $I-e_Je_J^\top$ reflects that the constant function, whose coefficient vector is $e_J$, has zero variance under every density. The second term is the covariance of the Laplace noise.

\begin{lem}
\label{lem:release-covariance}
For every $J>j_0$ and every density $f$ satisfying $\|f\|_\infty\leq U$,
$$\Cov_f^Q(Z_{i,J})\preceq\mathsf V_J.$$
\end{lem}

\begin{proof}
Let $c=(c_a)_{a\in\calI_J}$ and $h=\sum_{a\in\calI_J}c_a\psi_a$, so that $c^\top Z_{i,J}=h(X_i)+\sum_{a\in\calI_J}c_a\sigma_{J,a}W_{i,a}$. The Laplace variables are independent of $X_i$ and of one another, with variance two, so
$$c^\top\Cov_f^Q(Z_{i,J})\,c=\Var_f\{h(X_i)\}+2c^\top\operatorname{diag}_{a\in\calI_J}(\sigma_{J,a}^2)\,c.$$
A variance is at most the second moment about any constant. Taking the constant $\int_0^1h=\langle h,1\rangle=c^\top e_J$, and using $f\leq U$ and (W1), one obtains
$$
\begin{aligned}
\Var_f\{h(X_i)\}
&\leq U\int_0^1\{h(x)-c^\top e_J\}^2\,dx\\
&=U\{\|h\|_2^2-(c^\top e_J)^2\}\\
&=Uc^\top(I-e_Je_J^\top)c.
\end{aligned}
$$
Hence $c^\top\Cov_f^Q(Z_{i,J})\,c\leq c^\top\mathsf V_Jc$ for every $c$, which is the claim.
\end{proof}

We now turn to the projected-loss statistic itself. As in the projected-loss construction of \citet{gineMathematicalFoundationsInfiniteDimensional2016}, sample splitting allows us to condition on the estimator constructed from the first split and treat as fixed when analysing the second split.

Let $g\in L^2([0,1])$ be computed from the first split, so that it is independent of the reports $Z_{1,J},\ldots,Z_{m,J}$ of the second split, and recall that $\E_2$ and $\Var_2$ denote conditional moments given the first split. The statistic of Section~\ref{sec:construction} is
$$U_m(g,J)=\frac{1}{m(m-1)}\sum_{i\neq i'}\langle Z_{i,J}-\theta_J(g),Z_{i',J}-\theta_J(g)\rangle.$$
Its target is the squared distance between $f$ and $g$ after projection by $K_J$,
$$q_J(g)^2=\|\mu\|_{\ell^2}^2=\|K_J(f-g)\|_2^2,\qquad \mu=\theta_J(f)-\theta_J(g),$$
where the second identity follows from (W1). The next lemma shows that $U_m(g,J)$ is conditionally unbiased for $q_J(g)^2$ and bounds its conditional variance through $v_J=\|\mathsf V_J\|_{\op}/m$ and $w_J^2=2\operatorname{tr}(\mathsf V_J^2)/\{m(m-1)\}$.

\begin{lem}
\label{lem:projected-loss-moments}
For every density $f$ satisfying $\|f\|_\infty\leq U$, one has $\E_2U_m(g,J)=q_J(g)^2$,
$$\Var_2\{U_m(g,J)\}\leq4v_Jq_J(g)^2+w_J^2,$$
and
$$\E_2U_m(g,J)^+\leq q_J(g)^2+\frac12\sqrt{4v_Jq_J(g)^2+w_J^2}.$$
\end{lem}

\begin{proof}
Write $\mathsf C_J=\Cov_f^Q(Z_{i,J})$ and $\chi_i=Z_{i,J}-\theta_J(f)$ for $i=1,\ldots,m$. The vectors $\chi_1,\ldots,\chi_m$ are independent of the first split and of one another, centred, with covariance $\mathsf C_J$, and $Z_{i,J}-\theta_J(g)=\mu+\chi_i$. Substituting this identity in the definition of $U_m(g,J)$, one obtains
$$U_m(g,J)=q_J(g)^2+\Lambda+\Omega,$$
where
$$\Lambda=\frac2m\sum_{i=1}^m\mu^\top\chi_i,\qquad\Omega=\frac1{m(m-1)}\sum_{i\neq i'}\chi_i^\top\chi_{i'}$$
are the linear and degenerate parts of the statistic. Both are conditionally centred, the second because $\chi_i$ and $\chi_{i'}$ are independent and centred for $i\neq i'$. Hence $\E_2U_m(g,J)=q_J(g)^2$.

We next compute the conditional variance. The two parts are conditionally uncorrelated: $\E_2(\Lambda\Omega)$ is a combination of terms $\E_2\{(\mu^\top\chi_i)(\chi_{i_1}^\top\chi_{i_2})\}$ with $i_1\neq i_2$, and in each of them one of $\chi_{i_1}$ and $\chi_{i_2}$ appears only once, so that independence and centring make the term vanish. By independence,
$$
\begin{aligned}
\Var_2(\Lambda)&=\frac4{m^2}\sum_{i=1}^m\Var_2(\mu^\top\chi_i)\\
&=\frac4m\mu^\top\mathsf C_J\mu.
\end{aligned}
$$
In the expansion
$$
\begin{aligned}
\E_2\Omega^2&=\frac1{m^2(m-1)^2}\\
&\quad\times\sum_{i_1\neq i_2}\sum_{i_3\neq i_4}\E_2\{(\chi_{i_1}^\top\chi_{i_2})(\chi_{i_3}^\top\chi_{i_4})\},
\end{aligned}
$$
a term vanishes as soon as one of the indices appears only once, that is, unless $\{i_3,i_4\}=\{i_1,i_2\}$. In that case, integrating first over $\chi_{i_2}$,
$$\E_2(\chi_{i_1}^\top\chi_{i_2})^2=\E_2(\chi_{i_1}^\top\mathsf C_J\chi_{i_1})=\operatorname{tr}(\mathsf C_J^2).$$
Each of the $m(m-1)$ ordered pairs $(i_1,i_2)$ has two such partners, so that $\E_2\Omega^2=2\operatorname{tr}(\mathsf C_J^2)/\{m(m-1)\}$, and therefore
$$\Var_2\{U_m(g,J)\}=\frac4m\mu^\top\mathsf C_J\mu+\frac{2\operatorname{tr}(\mathsf C_J^2)}{m(m-1)}.$$

By Lemma~\ref{lem:release-covariance}, $0\preceq\mathsf C_J\preceq\mathsf V_J$. Hence $\mu^\top\mathsf C_J\mu\leq\|\mathsf V_J\|_{\op}\|\mu\|_{\ell^2}^2$, and
$$\operatorname{tr}(\mathsf V_J^2)-\operatorname{tr}(\mathsf C_J^2)=\operatorname{tr}\{(\mathsf V_J-\mathsf C_J)(\mathsf V_J+\mathsf C_J)\}\geq0,$$
since the trace of a product of two positive semidefinite matrices is non-negative. By the definitions of $v_J$, $w_J$ and $q_J(g)$, this gives $\Var_2\{U_m(g,J)\}\leq4v_Jq_J(g)^2+w_J^2$.

Finally, since $q_J(g)^2\geq0$, one has $U_m(g,J)^+\leq q_J(g)^2+(\Lambda+\Omega)^+$. As $\Lambda+\Omega$ is conditionally centred, its positive and negative parts have the same conditional expectation, so that, by the Cauchy-Schwarz inequality,
$$\E_2(\Lambda+\Omega)^+=\frac12\E_2|\Lambda+\Omega|\leq\frac12\bigl[\Var_2\{U_m(g,J)\}\bigr]^{1/2}.$$
The variance bound completes the proof.
\end{proof}

\subsection{Upper bound for the projected loss}

Lemma \ref{lem:projected-loss-moments} shows that $U_m(g,J)$ is centred at the squared projected loss $q_J(g)^2$ and that its conditional variance grows at most linearly with this unknown quantity. We now invert this variance bound to obtain an observable upper bound for $q_J(g)^2$. This is the step that turns the projected-loss statistic into the confidence radius used in Section \ref{sec:construction}. The argument is related to the variance-based control of private quadratic statistics in \citet{lam-weilMinimaxOptimalGoodnessoffit2021}, but here we solve the resulting inequality explicitly because the unknown projected loss itself appears in the variance.

Recall from Section~\ref{sec:construction} that, for $\eta\in(0,1)$ and $u\geq0$, with $c_\eta=(1-\eta)/\eta$,
$$
\begin{aligned}
\mathcal H_{\eta,J}(u)&=u+2c_\eta v_J\\
&\quad+\sqrt{4c_\eta v_Ju+4c_\eta^2v_J^2+c_\eta w_J^2},
\end{aligned}
$$
and $F_\eta(J)=\mathcal H_{\eta,J}(0)$. The bound is applied to the positive part $U_m(g,J)^+$, since $q_J(g)^2$ is non-negative whereas the statistic may be negative.

\begin{prop}
\label{prop:projected-loss-upper}
For every $\gamma\in(0,1)$, conditionally on the data used to construct $g$ and uniformly over $f\in\Sigma(r,R,U)$,
$$P_2\left[q_J(g)^2\leq\mathcal H_{\gamma,J}\{U_m(g,J)^+\}\right]\geq1-\gamma\quad\text{almost surely}.$$
Moreover, for every $\eta\in(0,1)$ and $u\geq0$,
$$\mathcal H_{\eta,J}(u)\leq2u+F_\eta(J).$$
\end{prop}

\begin{proof}
Fix $\gamma\in(0,1)$ and write
$$q^2=q_J(g)^2,\qquad U=U_m(g,J).$$
By Lemma \ref{lem:projected-loss-moments},
$$\E_2U=q^2,\qquad \Var_2(U)\leq4v_Jq^2+w_J^2.$$
Cantelli's inequality states that, for any centred random variable $Y$ with finite variance and any $a>0$,
$$P(Y\leq-a)\leq\frac{\Var(Y)}{\Var(Y)+a^2}.$$
Apply this with $Y=U-q^2$ and
$$a=\sqrt{c_\gamma(4v_Jq^2+w_J^2)},\qquad c_\gamma=\frac{1-\gamma}{\gamma}.$$
Since $\Var_2(U)\leq4v_Jq^2+w_J^2$,
$$
\begin{aligned}
&P_2\left\{
q^2-U>\sqrt{c_\gamma(4v_Jq^2+w_J^2)}
\right\}\\
&\quad\leq
\frac{4v_Jq^2+w_J^2}
{4v_Jq^2+w_J^2+c_\gamma(4v_Jq^2+w_J^2)}\\
&\quad=
\frac{1}{1+c_\gamma}
=\gamma.
\end{aligned}
$$
Therefore, with conditional probability at least $1-\gamma$,
$$q^2\leq U+\sqrt{c_\gamma(4v_Jq^2+w_J^2)}.$$
Since $U\leq U^+$, the same event implies
$$q^2\leq u+\sqrt{c_\gamma(4v_Jq^2+w_J^2)},\qquad u=U^+.$$

If $q^2\leq u$, then $q^2\leq\mathcal H_{\gamma,J}(u)$ immediately because $\mathcal H_{\gamma,J}(u)\geq u$. Suppose instead that $q^2>u$. Both sides of
$$q^2-u\leq\sqrt{c_\gamma(4v_Jq^2+w_J^2)}$$
are then non-negative, so we may square to obtain
$$(q^2-u)^2\leq4c_\gamma v_Jq^2+c_\gamma w_J^2.$$
Writing $y=q^2$, this is equivalent to
$$y^2-(2u+4c_\gamma v_J)y+u^2-c_\gamma w_J^2\leq0.$$
The larger root of the corresponding quadratic equation is
$$\begin{aligned}y_+&=\frac{2u+4c_\gamma v_J+\sqrt{(2u+4c_\gamma v_J)^2-4(u^2-c_\gamma w_J^2)}}{2}\\&=u+2c_\gamma v_J+\sqrt{4c_\gamma v_Ju+4c_\gamma^2v_J^2+c_\gamma w_J^2}\\&=\mathcal H_{\gamma,J}(u).\end{aligned}$$
Since the quadratic is non-positive only between its two roots, we have $q^2=y\leq y_+$. Hence
$$P_2\left[q_J(g)^2\leq\mathcal H_{\gamma,J}\{U_m(g,J)^+\}\right]\geq1-\gamma.$$
The bound is uniform over $f\in\Sigma(r,R,U)$ because the variance bound in Lemma \ref{lem:projected-loss-moments} is uniform over densities satisfying $\|f\|_\infty\leq U$.

It remains to prove the deterministic bound on $\mathcal H_{\eta,J}$. For $u\geq0$,
$$\frac{\partial}{\partial u}\mathcal H_{\eta,J}(u)=1+\frac{2c_\eta v_J}{\sqrt{4c_\eta v_Ju+4c_\eta^2v_J^2+c_\eta w_J^2}}.$$
The denominator is at least $2c_\eta v_J$, so
$$\frac{\partial}{\partial u}\mathcal H_{\eta,J}(u)\leq2.$$
Therefore
$$\mathcal H_{\eta,J}(u)-\mathcal H_{\eta,J}(0)=\int_0^u\frac{\partial}{\partial s}\mathcal H_{\eta,J}(s)\,ds\leq2u.$$
Since $F_\eta(J)=\mathcal H_{\eta,J}(0)$, this gives
$$\mathcal H_{\eta,J}(u)\leq2u+F_\eta(J).$$
\end{proof}

\subsection{Asymptotic diameter}

We now determine the order of the deterministic part of the confidence radius. The projection level $J$ must balance the approximation error $b_r(J)$ against the uncertainty in estimating the projected loss. As in the non-private projected-loss construction of \citet{gineMathematicalFoundationsInfiniteDimensional2016}, the ordinary sampling contribution has order $2^{J/2}/m$. Under local privacy, there is an additional contribution of order $2^{3J/2}/(m\alpha^2)$, as in the private quadratic-statistic analysis of \citet{lam-weilMinimaxOptimalGoodnessoffit2021}. Our level-dependent wavelet release introduces the additional factor $(J\vee1)^{2\nu}$ because the noise scale at level $j$ is proportional to $(j\vee1)^\nu2^{j/2}/\alpha$.

Recall that
$$\Gamma_{m,\alpha}(J)=\frac{2^{J/2}}{m}+\frac{(J\vee1)^{2\nu}2^{3J/2}}{m\alpha^2},$$
and define
$$\rho_{m,\test}(r,\alpha)=m^{-2r/(4r+1)}\vee(m\alpha^2)^{-2r/(4r+3)}.$$
The first branch is obtained by balancing the Besov approximation error with the ordinary sampling term. The second is obtained by balancing it with the private term. The next lemma makes these bounds precise and shows that the deterministic resolution $J^\star$ chosen in Section \ref{sec:construction} attains the required order up to the logarithmic factor inherited from the multilevel private release.

\begin{lem}
\label{lem:finite-resolution-rates}
Under bounded privacy $\alpha\leq \bar \alpha$, for every $J>j_0$,
$$v_J\lesssim\frac1m+\frac{(J\vee1)^{2\nu}2^J}{m\alpha^2},\qquad w_J\lesssim\Gamma_{m,\alpha}(J),$$
and, for fixed $\gamma\in(0,1)$,
$$F_\gamma(J)\asymp_\gamma v_J+w_J\lesssim\Gamma_{m,\alpha}(J).$$
Moreover, the minimiser $J^\star$ of $b_r+F_\gamma$ exists, and, as $m\alpha^2\to\infty$,
$$
\begin{aligned}
b_r(J^\star)+F_\gamma(J^\star)
&\lesssim_\gamma\inf_{J>j_0}\{b_r(J)+\Gamma_{m,\alpha}(J)\}\\
&\lesssim\ell_m^2\rho_{m,\test}(r,\alpha)^2,
\end{aligned}
$$
where $\rho_{m,\test}(r,\alpha)$ is the quantity $\rho_{n,\test}(r,\alpha)$ of Section~\ref{sec:model} with $n$ replaced by $m$.
\end{lem}

\begin{proof}
We bound $v_J$ and $w_J$ through the operator and Hilbert-Schmidt norms of $\mathsf V_J=U(I-e_Je_J^\top)+2\operatorname{diag}_{a\in\calI_J}(\sigma_{J,a}^2)$, deduce the bounds on $F_\gamma(J)$, and then evaluate the criterion defining $J^\star$ at a suitable resolution. Recall from Lemma~\ref{lem:release-moments} that $\sigma_{J,a}^2\lesssim\alpha^{-2}d_{j(a)}^2$, where $d_{-1}^2=2^{j_0}$ and $d_{j(a)}^2=(j\vee1)^{2\nu}2^j$ at level $j<J$ otherwise.

By (W1) and (W4), $e_J$ is a unit vector, so that $I-e_Je_J^\top$ is an orthogonal projection, with operator norm one and rank $|\calI_J|-1=2^J-1$. Hence
$$\|\mathsf V_J\|_{\op}\leq U+2\max_{a\in\calI_J}\sigma_{J,a}^2\lesssim1+\frac{(J\vee1)^{2\nu}2^J}{\alpha^2},$$
which gives the bound on $v_J=\|\mathsf V_J\|_{\op}/m$. Similarly, since $\sum_{a\in\calI_J}d_{j(a)}^4\lesssim1+\sum_{j=j_0}^{J-1}(j\vee1)^{4\nu}2^{3j}\lesssim(J\vee1)^{4\nu}2^{3J}$,
$$
\begin{aligned}
\|\mathsf V_J\|_{\HS}&\leq U(2^J-1)^{1/2}+2\Bigl(\sum_{a\in\calI_J}\sigma_{J,a}^4\Bigr)^{1/2}\\
&\lesssim2^{J/2}+\frac{(J\vee1)^{2\nu}2^{3J/2}}{\alpha^2}.
\end{aligned}
$$
As $\operatorname{tr}(\mathsf V_J^2)=\|\mathsf V_J\|_{\HS}^2$ and $m(m-1)\geq m^2/2$, it follows that $w_J\leq2\|\mathsf V_J\|_{\HS}/m\lesssim\Gamma_{m,\alpha}(J)$.

Next, $F_\gamma(J)=2c_\gamma v_J+(4c_\gamma^2v_J^2+c_\gamma w_J^2)^{1/2}$, and the square root lies between $\max\{2c_\gamma v_J,\sqrt{c_\gamma}w_J\}$ and $2c_\gamma v_J+\sqrt{c_\gamma}w_J$. Hence
$$\max\{4c_\gamma v_J,2\sqrt{c_\gamma}w_J\}\leq F_\gamma(J)\leq4c_\gamma v_J+\sqrt{c_\gamma}w_J,$$
so that $F_\gamma(J)\asymp_\gamma v_J+w_J$. Since $1\leq2^{J/2}$ and $2^J\leq2^{3J/2}$, the bound on $v_J$ is also at most a constant multiple of $\Gamma_{m,\alpha}(J)$, and therefore $v_J+w_J\lesssim\Gamma_{m,\alpha}(J)$.

We now show that the minimum defining $J^\star$ is attained. Enlarging $\calI_J$ adds non-negative terms to the definition \eqref{eq: delta_j} of $\Delta_J$, so that $\Delta_J\geq\Delta_{j_0+1}>0$ for every $J>j_0$. The operator norm of $\mathsf V_J$ is at least its largest diagonal entry, which, at a coefficient of level $J-1$, is at least $2\sigma_{J,a}^2=2\alpha^{-2}\Delta_J^2d_{j(a)}^2\geq2\alpha^{-2}\Delta_{j_0+1}^2(J-1)^{2\nu}2^{J-1}$. Hence, for fixed $m$ and $\alpha$, $v_J\to\infty$, and so $F_\gamma(J)\to\infty$, as $J\to\infty$. Since $b_r\geq0$ and $J$ ranges over the integers larger than $j_0$, the minimum is attained.

It remains to bound the infimum, for which we may assume that $N_m=m\alpha^2$ is large. Put $H_m=\log(e+N_m)$, so that $\ell_m^2=H_m^{2\nu+1}$, and choose an integer $J^\circ>j_0$ with
$$2^{J^\circ}\asymp\Bigl(\frac{N_m}{H_m^{2\nu}}\Bigr)^{2/(4r+3)},$$
so that $J^\circ\lesssim H_m$. At this resolution, the approximation term, and the private part of $\Gamma_{m,\alpha}(J^\circ)$ are of the same order:
$$b_r(J^\circ)\lesssim2^{-2J^\circ r}\asymp\Bigl(\frac{N_m}{H_m^{2\nu}}\Bigr)^{-4r/(4r+3)},$$
and
$$
\begin{aligned}
\frac{(J^\circ\vee1)^{2\nu}2^{3J^\circ/2}}{m\alpha^2}
&\lesssim\frac{H_m^{2\nu}}{N_m}\Bigl(\frac{N_m}{H_m^{2\nu}}\Bigr)^{3/(4r+3)}\\
&=\Bigl(\frac{N_m}{H_m^{2\nu}}\Bigr)^{-4r/(4r+3)}.
\end{aligned}
$$
The sampling part of $\Gamma_{m,\alpha}(J^\circ)$ is negligible, since its ratio to the private part is $\alpha^2/\{(J^\circ\vee1)^{2\nu}2^{J^\circ}\}\leq\bar\alpha^2/2^{J^\circ}\to0$. Since $8\nu r/(4r+3)<2\nu$ and $H_m\geq1$, one has $(N_m/H_m^{2\nu})^{-4r/(4r+3)}\leq H_m^{2\nu}N_m^{-4r/(4r+3)}$, and therefore
$$
\begin{aligned}
\inf_{J>j_0}\{b_r(J)+\Gamma_{m,\alpha}(J)\}
&\leq b_r(J^\circ)+\Gamma_{m,\alpha}(J^\circ)\\
&\lesssim\ell_m^2N_m^{-4r/(4r+3)}\\
&\leq\ell_m^2\rho_{m,\test}(r,\alpha)^2.
\end{aligned}
$$
Finally, since $J^\star$ minimises $b_r(J)+F_\gamma(J)$ and $F_\gamma(J)\lesssim_\gamma\Gamma_{m,\alpha}(J)$ for every $J$,
$$
\begin{aligned}
b_r(J^\star)+F_\gamma(J^\star)&=\inf_{J>j_0}\{b_r(J)+F_\gamma(J)\}\\
&\lesssim_\gamma\inf_{J>j_0}\{b_r(J)+\Gamma_{m,\alpha}(J)\}.
\end{aligned}
$$
\end{proof}

\begin{proof}[Proof of Theorem~\ref{thm:projected-loss}]
We prove the two assertions in turn, writing $J=J^\star$ throughout. The split assignments are deterministic, and each individual belongs to a single split. Individuals of the first split use the release of Lemma~\ref{lem:release-privacy} through level $K_{m_1}-1$, and those of the second split the same release through level $J-1$.

We first prove honesty. Let $f\in\Sigma(r,R,U)$. The centre $g_n=K_J\widetilde f_{m_1}$ is computed from the first split, so Proposition~\ref{prop:projected-loss-upper}, applied with $g=g_n$ and $m=m_2$, shows that the event
$$\|K_J(f-g_n)\|_2^2\leq\mathcal H_{\gamma,J}\{U_{m_2}(g_n,J)^+\}$$
has conditional probability at least $1-\gamma$ given the first split. Since $(I-K_J)g_n=0$, orthogonality and Lemma~\ref{lem:besov-tail} give
$$
\begin{aligned}
\|f-g_n\|_2^2&=\|K_J(f-g_n)\|_2^2+\|(I-K_J)f\|_2^2\\
&\leq\|K_J(f-g_n)\|_2^2+b_r(J),
\end{aligned}
$$
so that $\|f-g_n\|_2^2\leq R_n^2$ on that event with $R_n^2$ as \eqref{eq: Rn2}. As $f\in\Sigma(r,R,U)$, this means that $f\in C_n(\gamma)$. Taking expectations over the first split gives $P_f^{Q^n}\{f\in C_n(\gamma)\}\geq1-\gamma$ for every $f\in\Sigma(r,R,U)$ and every $n$.

It remains to bound the diameter. Every element of $C_n(\gamma)$ lies within distance $R_n$ of $g_n$, so that $|C_n(\gamma)|_2\leq2R_n$, and it suffices to bound $\E_f^{Q^n}R_n^2$. Fix $t\in[r,S)$ and $f\in\Sigma(t,R,U)\subseteq\Sigma(r,R,U)$. By the second part of Proposition~\ref{prop:projected-loss-upper},
$$R_n^2\leq b_r(J)+F_\gamma(J)+2U_{m_2}(g_n,J)^+,$$
and, by Lemma~\ref{lem:projected-loss-moments} applied given the first split, with $q_J(g_n)^2=\|K_J(f-g_n)\|_2^2$,
$$\E_2U_{m_2}(g_n,J)^+\leq q_J(g_n)^2+\frac12\sqrt{4v_Jq_J(g_n)^2+w_J^2}.$$
The projected loss depends only on the first split, and it is controlled by the risk of the estimator: since $K_J(f-g_n)=K_J(f-\widetilde f_{m_1})$ and $K_J$ is a contraction, Proposition~\ref{prop:adaptive-private-estimator} gives
$$\E_f^{Q^{m_1}}q_J(g_n)^2\leq\E_f^{Q^{m_1}}\|f-\widetilde f_{m_1}\|_2^2\leq a_{m_1}(t).$$
Combining the last three displays, using Jensen's inequality for the concave map $y\mapsto(4v_Jy+w_J^2)^{1/2}$, and then $\{4v_Ja_{m_1}(t)+w_J^2\}^{1/2}\leq2\{v_Ja_{m_1}(t)\}^{1/2}+w_J\leq a_{m_1}(t)+v_J+w_J$, we obtain
$$
\begin{aligned}
\E_f^{Q^n}R_n^2&\leq b_r(J)+F_\gamma(J)+3a_{m_1}(t)+v_J+w_J\\
&\lesssim_\gamma a_{m_1}(t)+b_r(J)+F_\gamma(J),
\end{aligned}
$$
the last step because $v_J+w_J\lesssim_\gamma F_\gamma(J)$ by Lemma~\ref{lem:finite-resolution-rates}. By Proposition~\ref{prop:adaptive-private-estimator} and Lemma~\ref{lem:finite-resolution-rates},
$$
\begin{aligned}
a_{m_1}(t)&\lesssim\ell_{m_1}^2\delta_{m_1,\est}(t,\alpha)^2,\\
b_r(J)+F_\gamma(J)&\lesssim_\gamma\ell_{m_2}^2\rho_{m_2,\test}(r,\alpha)^2.
\end{aligned}
$$
Since $\ell_m$ increases with $m$ and $n/3\leq m_1\leq m_2\leq n$, one has $\ell_{m_1}\vee\ell_{m_2}\leq\ell_n$, $\delta_{m_1,\est}(t,\alpha)\lesssim\delta_{n,\est}(t,\alpha)$ and $\rho_{m_2,\test}(r,\alpha)\lesssim\rho_{n,\test}(r,\alpha)$. Hence, uniformly over $f\in\Sigma(t,R,U)$,
$$
\begin{aligned}
\E_f^{Q^n}|C_n(\gamma)|_2^2&\leq4\E_f^{Q^n}R_n^2\\
&\lesssim_\gamma\ell_n^2\{\delta_{n,\est}(t,\alpha)\vee\rho_{n,\test}(r,\alpha)\}^2,
\end{aligned}
$$
and Markov's inequality gives, for every $M>0$,
$$
\begin{aligned}
&\sup_{f\in\Sigma(t,R,U)}P_f^{Q^n}\bigl[|C_n(\gamma)|_2\\
&\quad>M\ell_n\{\delta_{n,\est}(t,\alpha)\vee\rho_{n,\test}(r,\alpha)\}\bigr]\lesssim_\gamma M^{-2},
\end{aligned}
$$
which is the claimed uniform $O_{P_f^{Q^n}}$ bound.
\end{proof}

We can now conclude the positive side of the adaptation theorem. Theorem \ref{thm:projected-loss} already establishes honesty over the rough class and a diameter of order $\ell_n\{\delta_{n,\est}(t,\alpha)\vee\rho_{n,\test}(r,\alpha)\}$ on each class. It remains only to determine when the additional projected-loss term $\rho_{n,\test}(r,\alpha)$ is no larger than the corresponding estimation radius. This comparison is where the boundary $s=4r/3$ becomes present in the positive result.

\begin{proof}[Positive side of Theorem~\ref{thm:main-adaptation}]
Let $\mathcal R\subseteq[r,4r/3]$ have smallest element $r$, and consider the confidence set $C_n(\gamma)$ of Section~\ref{sec:construction}, which depends on $\mathcal R$ only through $r$. By Theorem~\ref{thm:projected-loss}, it is honest over $\Sigma(r,R,U)$ at level $1-\gamma$ for every $n$, and, for every $t\in\mathcal R$, its diameter is of order $\ell_n\{\delta_{n,\est}(t,\alpha)\vee\rho_{n,\test}(r,\alpha)\}$ over $\Sigma(t,R,U)$. It therefore suffices to show that $\rho_{n,\test}(r,\alpha)\lesssim\delta_{n,\est}(t,\alpha)$ for every $t\in\mathcal R$. Under bounded privacy, $\rho_{n,\test}(r,\alpha)\asymp(n\alpha^2)^{-2r/(4r+3)}$ and $\delta_{n,\est}(t,\alpha)\asymp(n\alpha^2)^{-t/(2t+2)}$. The map $t\mapsto t/(2t+2)$ is increasing and takes the value $2r/(4r+3)$ at $t=4r/3$, so that $t/(2t+2)\leq2r/(4r+3)$ for every $t\in\mathcal R$. Since $n\alpha^2\geq1$ for $n$ large enough, it follows that
$$\rho_{n,\test}(r,\alpha)\lesssim(n\alpha^2)^{-t/(2t+2)}\asymp\delta_{n,\est}(t,\alpha),\qquad t\in\mathcal R,$$
which proves the claim.
\end{proof}

\section{Lower bounds for smoothness testing and adaptation}
\label{app:testing-lower-impossibility}

This appendix proves Theorem~\ref{thm:testing-lower} and the negative side of Theorem~\ref{thm:main-adaptation}. Throughout this appendix, $r<s$ are two regularities in $(0,S)$. The argument has two parts. First, as in the non-private theory of adaptive confidence sets \citep{lowNonparametricConfidenceIntervals1997,CaiLow2006,robinsAdaptiveNonparametricConfidence2006,bullAdaptiveConfidenceSets2012,gineMathematicalFoundationsInfiniteDimensional2016}, an honest confidence set whose diameter is small on the smooth class induces a test of smoothness (Appendix~\ref{app:reduction}). Second, no locally private test can distinguish the smooth class from rough densities that lie closer to it than the private testing rate. For this second part, we adapt the channel-dependent prior that \citet{lam-weilMinimaxOptimalGoodnessoffit2021} use for goodness-of-fit testing of a simple null. Since our null hypothesis consists of the whole smooth class $\Sigma(s,R,U)$, the alternatives we consisder must in addition be separated from every such smooth density, as in the non-private composite lower bounds of \citet{bullAdaptiveConfidenceSets2012} and \citet{carpentierTestingRegularitySmooth2015}; this is where our argument departs from the simple-null case.

Throughout, $H_1(\rho)$ denotes the set of alternatives of Section~\ref{sec:smoothness-testing}, namely the densities of $\Sigma(r,R,U)$ at $L^2$ distance at least $\rho$ from $\Sigma(s,R,U)$. Non-interactive channels randomise independently across individuals, as in Section~\ref{sec:model}, and this product structure is used in Lemma~\ref{lem:conditioned-prior-second-moment}.

\subsection{Confidence sets imply smoothness tests}
\label{app:reduction}

A set that intersects with the smooth class and has diameter at most $v_n$ cannot contain a density at distance more than $v_n$ from that class. The test below therefore rejects exactly when the set misses the smooth class or is large, and honesty controls both of its errors. We record this reduction in finite-sample form.

\begin{prop}
\label{prop:testing-reduction}
Let $C_n$ be honest over $\Sigma(r,R,U)$ at level $1-\gamma$, let $v_n>0$ be deterministic, and let
$$\Psi_n=\1\{C_n\cap\Sigma(s,R,U)=\varnothing\ \text{or}\ |C_n|_2>v_n\}.$$
Then
$$
\begin{aligned}
&\sup_{f\in\Sigma(s,R,U)}P_f^{Q^n}(\Psi_n=1)\\
&\quad\leq\gamma+\sup_{f\in\Sigma(s,R,U)}P_f^{Q^n}(|C_n|_2>v_n),
\end{aligned}
$$
and $P_f^{Q^n}(\Psi_n=0)\leq\gamma$ for every $f\in\Sigma(r,R,U)$ with $\dist_2(f,\Sigma(s,R,U))>v_n$.
\end{prop}

\begin{proof}
Let $f\in\Sigma(s,R,U)$. If $f\in C_n$ and $|C_n|_2\leq v_n$, then $C_n$ meets $\Sigma(s,R,U)$ at $f$, so that $\Psi_n=0$. Hence $\{\Psi_n=1\}\subseteq\{f\notin C_n\}\cup\{|C_n|_2>v_n\}$, and the first bound follows from a union bound and honesty of $C_n$. Now let $f\in\Sigma(r,R,U)$ with $\dist_2(f,\Sigma(s,R,U))>v_n$. If $\Psi_n=0$ and $f\in C_n$, then $C_n$ contains both $f$ and some $h\in\Sigma(s,R,U)$, so that $\dist_2(f,\Sigma(s,R,U))\leq\|f-h\|_2\leq|C_n|_2\leq v_n$, which is impossible. Hence $\{\Psi_n=0\}\subseteq\{f\notin C_n\}$, and the second bound also follows from honesty.
\end{proof}

\subsection{A channel-adapted prior}
\label{app:prior}

We construct the alternatives used in the lower bound proof in Appendix \ref{app:lower-bound-proofs} as perturbations of the uniform density, which we denote by $1$, at a single wavelet resolution level $J$. Following \citet{lam-weilMinimaxOptimalGoodnessoffit2021} (whose result we adapt to a composite null and our boundary-corrected wavelet basis), the perturbation directions are adapted to the channels: directions that the channels transmit poorly receive larger amplitudes, so that the alternatives remain far from the smooth class before privatisation while their induced probability distribution over private summaries stay close to that from the uniform density. 

Fix non-interactive $\alpha$-LDP channels $Q_1,\ldots,Q_n$, and let
$$\mu_i(A)=\int_0^1Q_i(A\mid x)\,dx$$
be the law of the $i$th report under the uniform density.

\begin{lem}
\label{lem:channel-domination}
For every $i$ and every $x\in[0,1]$, one has $e^{-\alpha}\mu_i\leq Q_i(\cdot\mid x)\leq e^\alpha\mu_i$. In particular, $Q_i(\cdot\mid x)$ has a density $q_i(\cdot\mid x)$ with respect to $\mu_i$, which satisfies $e^{-\alpha}\leq q_i(\cdot\mid x)\leq e^\alpha$ $\mu_i$-almost everywhere.
\end{lem}

\begin{proof}
For every measurable set $A$, integrating $Q_i(A\mid x)\leq e^\alpha Q_i(A\mid x')$ over $x'\in[0,1]$ gives $Q_i(A\mid x)\leq e^\alpha\mu_i(A)$, and integrating $Q_i(A\mid x')\leq e^\alpha Q_i(A\mid x)$ gives $\mu_i(A)\leq e^\alpha Q_i(A\mid x)$. The claims follow from the Radon-Nikodym theorem.
\end{proof}

We use a version of $q_i$ that is jointly measurable in $(z,x)$, which exists, for instance, when the output space is countably generated (as is the case for most LDP mechanisms considered in the literature). Since $\int_A\int_0^1q_i(z\mid x)\,dx\,d\mu_i(z)=\int_0^1Q_i(A\mid x)\,dx=\mu_i(A)$ for every measurable $A$, Fubini's theorem gives \[\int_0^1q_i(z\mid x)\,dx=1\] for $\mu_i$-almost every $z$. For $u\in L^2([0,1])$, put
$$\mathcal T_iu(z)=\int_0^1q_i(z\mid x)u(x)\,dx.$$
If $1+u$ is a density, then $1+\mathcal T_iu$ is the induced density with respect to $\mu_i$ of the $i$th report under the $1+u$, so that $\mathcal T_i$ describes how the $i$th channel transmits a perturbation of the uniform density.

Let $\calV_J$ be the span of the wavelets $\psi_{Jk}$, $k\in N_J$. By (W2) and (W4), it has dimension $L=2^J$ and its elements have mean zero (they are orthogonal to the uniform density, identically equal to 1). On $\calV_J$, consider the symmetric non-negative bilinear form
$$\langle u,u'\rangle_Q=\frac1n\sum_{i=1}^n\int\mathcal T_iu\,\mathcal T_iu'\,d\mu_i,$$
which measures how much information the reports carry, on average, about the directions $u$ and $u'$. Let $u_1,\ldots,u_L$ be an orthonormal basis of $\calV_J$ (for the usual euclidean structure) that diagonalises this form, so that $\langle u_j,u_l\rangle_Q=\lambda_j\1\{j=l\}$ with $\lambda_j\geq0$, and put
\[\widetilde\lambda_j=\frac{\lambda_j}{e^{2\bar\alpha}\alpha^2}\vee\frac1L.\]
For a sign vector $\omega\in\{-1,1\}^L$, we introduce the alternatives
$$f_\omega=1+L^{-r-1}\sum_{j=1}^L\omega_j\widetilde\lambda_j^{-1/2}u_j.$$
The factor $\widetilde\lambda_j^{-1/2}$ enlarges the directions that the channels transmit poorly, the truncation at $L^{-1}$ keeps every amplitude bounded, and the factor $L^{-r-1}$ makes $f_\omega-1$ of size $L^{-r}$ (Lemma~\ref{lem:alternative-integral-regularity}). The next lemma is where local privacy enters: it limits the total information $\sum_j\lambda_j$.

\begin{lem}
\label{lem:spectral-consequences}
One has $\sum_{j=1}^L\lambda_j\leq(e^\alpha-1)^2\leq e^{2\bar \alpha}\alpha^2$. Consequently, $\lambda_j/\widetilde\lambda_j\leq e^{2\bar \alpha}\alpha^2$ for every $j$, and
$$\frac{L^2}2\leq\sum_{j=1}^L\widetilde\lambda_j^{-1}\leq L^2.$$
\end{lem}

\begin{proof}
By Lemma~\ref{lem:channel-domination} and Fubini's theorem, for $\mu_i$-almost every $z$, one has $e^{-\alpha}\leq q_i(z\mid x)\leq e^\alpha$ for almost every $x$, and hence $\|q_i(z\mid\cdot)-1\|_2\leq e^\alpha-1$. For such $z$, since each $u_j$ has mean zero, $\mathcal T_iu_j(z)=\langle q_i(z\mid\cdot)-1,u_j\rangle$, and Bessel's inequality gives
$$\sum_{j=1}^L\{\mathcal T_iu_j(z)\}^2\leq\|q_i(z\mid\cdot)-1\|_2^2\leq(e^\alpha-1)^2.$$
Integrating against the probability measure $\mu_i$ and averaging over $i$ yields 
\begin{align*}
\sum_j\lambda_j&=\sum_j\langle u_j,u_j\rangle_Q\\
&=\frac1n\sum_{i=1}^n\int\sum_j(\mathcal T_iu_j)^2\,d\mu_i\\
&\leq(e^\alpha-1)^2,
\end{align*}
and $e^\alpha-1\leq\alpha e^{\alpha}\leq\alpha e^{\bar\alpha}$. Next, $\lambda_j/\widetilde\lambda_j\leq e^{2\bar \alpha}\alpha^2$ because $\widetilde\lambda_j\geq\lambda_j/(e^{2\bar\alpha}\alpha^2)$, and $\widetilde\lambda_j^{-1}\leq L$ because $\widetilde\lambda_j\geq L^{-1}$, which gives the upper bound on $\sum_j\widetilde\lambda_j^{-1}$. Finally, $\sum_j\widetilde\lambda_j\leq\sum_j\left(\lambda_j/(e^{2\bar\alpha}\alpha^2)+L^{-1}\right)\leq2$, so that, by the Cauchy-Schwarz inequality,
$$L^2\leq\Bigl(\sum_{j=1}^L\widetilde\lambda_j\Bigr)\Bigl(\sum_{j=1}^L\widetilde\lambda_j^{-1}\Bigr)\leq2\sum_{j=1}^L\widetilde\lambda_j^{-1}.$$
\end{proof}

\subsection{Validity of the alternatives}
\label{app:validity}

The alternatives we introduced above must be densities in $\Sigma(r,R,U)$ that are far from the whole smooth class. While integration to 1, the Besov constraint and the separation hold for every sign vector, positivity and the bound by $U$ hold on a set of sign vectors of high probability, on which the prior we consider will be conditioned. In contrast with the Haar wavelets of \citet{lam-weilMinimaxOptimalGoodnessoffit2021}, our wavelets need not have disjoint supports, and the sup-norm bound of Lemma~\ref{lem:wavelet-basic-bounds} replaces the disjointness argument they used.

\begin{lem}
\label{lem:alternative-integral-regularity}
For every $\omega\in\{-1,1\}^L$, the function $f_\omega$ integrates to one, and its wavelet coefficients coincide with those of the uniform density except at level $J$, where
$$\frac{L^{-r}}{\sqrt2}\leq\|\beta_J(f_\omega)\|_{\ell^2}=\|f_\omega-1\|_2\leq L^{-r}.$$
Consequently, $\|f_\omega\|_{B^r_{2,\infty}}\leq R$ and
$$\dist_2(f_\omega,\Sigma(s,R,U))\geq\frac{L^{-r}}{\sqrt2}-RL^{-s}.$$
\end{lem}

\begin{proof}
The function $f_\omega-1$ lies in $\calV_J$, so it has mean zero and its only non-zero wavelet coefficients are at level $J$. By orthonormality of the $u_j$, $\|f_\omega-1\|_2^2=L^{-2r-2}\sum_j\widetilde\lambda_j^{-1}$, which gives the displayed bounds by Lemma~\ref{lem:spectral-consequences}. The coefficients of $f_\omega$ at level $-1$ are those of the uniform density, whose Euclidean norm is one by (W1) and (W4). Since $2^{Jr}=L^r$, it follows that $\|f_\omega\|_{B^r_{2,\infty}}=\max\{1,L^r\|f_\omega-1\|_2\}\leq1\leq R$. Finally, for $h\in\Sigma(s,R,U)$, one has $\|\beta_J(h)\|_{\ell^2}\leq R2^{-Js}=RL^{-s}$, so that, by (W1) and the reverse triangle inequality,
$$\|f_\omega-h\|_2\geq\|\beta_J(f_\omega)-\beta_J(h)\|_{\ell^2}\geq\frac{L^{-r}}{\sqrt2}-RL^{-s},$$
and the bound on the distance follows by taking the infimum over $h$.
\end{proof}

Let $\pi$ be the prior uniform distribution on $\{-1,1\}^L$ and, for $p\in(0,1)$, let $\mathcal A_{L,p}$ be the set of sign vectors $\omega$ such that
$$\max_{k\in N_J}|\beta_{Jk}(f_\omega)|\leq L^{-r-1/2}\{2\log(2L/p)\}^{1/2}.$$

\begin{lem}
\label{lem:probabilistic-positivity}
For every $p\in(0,1)$, one has $\pi(\mathcal A_{L,p})\geq1-p$, and every $\omega\in\mathcal A_{L,p}$ satisfies
$$\|f_\omega-1\|_\infty\leq C_\psi L^{-r}\{2\log(2L/p)\}^{1/2}.$$
In particular, $f_\omega\in\Sigma(r,R,U)$ for every $\omega\in\mathcal A_{L,p}$ as soon as the right-hand side is at most $1\wedge(U-1)$.
\end{lem}

\begin{proof}
Since $f_\omega-1$ lies in $\calV_J$, one has $f_\omega-1=\sum_{k\in N_J}\beta_{Jk}(f_\omega)\psi_{Jk}$, where, for every $k\in N_J$,
$$\beta_{Jk}(f_\omega)=L^{-r-1}\sum_{j=1}^L\omega_j\widetilde\lambda_j^{-1/2}\langle u_j,\psi_{Jk}\rangle.$$
Under $\pi$, this is a sum of the independent Rademacher variables $\omega_1,\ldots,\omega_L$ with deterministic weights, whose squares sum to
$$
\begin{aligned}
L^{-2r-2}\sum_{j=1}^L\widetilde\lambda_j^{-1}\langle u_j,\psi_{Jk}\rangle^2
&\leq L^{-2r-1}\sum_{j=1}^L\langle u_j,\psi_{Jk}\rangle^2\\
&=L^{-2r-1},
\end{aligned}
$$
because $\widetilde\lambda_j^{-1}\leq L$ and, by Parseval's identity in $\calV_J$, $\sum_j\langle u_j,\psi_{Jk}\rangle^2=\|\psi_{Jk}\|_2^2=1$. For the level $L^{-r-1/2}\{2\log(2L/p)\}^{1/2}$ that defines $\mathcal A_{L,p}$, whose square is $2L^{-2r-1}\log(2L/p)$, Hoeffding's inequality gives
$$
\begin{aligned}
&\pi\Bigl[|\beta_{Jk}(f_\omega)|>L^{-r-1/2}\{2\log(2L/p)\}^{1/2}\Bigr]\\
&\quad\leq2\exp\Bigl\{-\frac{2L^{-2r-1}\log(2L/p)}{2L^{-2r-1}}\Bigr\}=\frac pL.
\end{aligned}
$$
The complement $\mathcal A_{L,p}^c$ is the union of these events over the $L$ indices $k\in N_J$, so that a union bound gives $\pi(\mathcal A_{L,p}^c)\leq p$.

On $\mathcal A_{L,p}$, Lemma~\ref{lem:wavelet-basic-bounds} and $2^{J/2}=L^{1/2}$ give
$$
\begin{aligned}
\|f_\omega-1\|_\infty
&\leq C_\psi2^{J/2}\max_{k\in N_J}|\beta_{Jk}(f_\omega)|\\
&\leq C_\psi L^{1/2}\cdot L^{-r-1/2}\{2\log(2L/p)\}^{1/2}\\
&=C_\psi L^{-r}\{2\log(2L/p)\}^{1/2}.
\end{aligned}
$$
If this bound is at most $1\wedge(U-1)$, then $0\leq f_\omega\leq U$, and Lemma~\ref{lem:alternative-integral-regularity} completes the proof.
\end{proof}

\subsection{Second moment of the private mixture}
\label{app:second-moment}

We condition the prior distribution on the event of Lemma~\ref{lem:probabilistic-positivity}, so that we obtain a mixture of densities, supported on valid alternatives. \citet{lam-weilMinimaxOptimalGoodnessoffit2021} instead keep the full prior and allow part of its mass outside the class of alternatives. Conditioning only costs the factor $\pi(\mathcal A)^{-2}$ in the next lemma.

\begin{lem}
\label{lem:conditioned-prior-second-moment}
Let $\mathcal A\subseteq\{-1,1\}^L$ satisfy $\pi(\mathcal A)\geq1-p$, and suppose that $f_\omega$ is a density for every $\omega\in\mathcal A$. Let $\Pi$ be the law of $f_\omega$ when $\omega$ is drawn from $\pi(\cdot\mid\mathcal A)$, and let $\mathcal L_n=dP_\Pi^{Q^n}/dP_1^{Q^n}$. Then
$$\E_1^{Q^n}\mathcal L_n^2\leq(1-p)^{-2}\exp\Bigl(\frac{n^2e^{4\bar\alpha}\alpha^4}{2L^{4r+3}}\Bigr).$$
\end{lem}

\begin{proof}
For $\omega\in\{-1,1\}^L$ and each $i$, put
$$\mathcal L_{\omega,i}=1+L^{-r-1}\sum_{j=1}^L\omega_j\widetilde\lambda_j^{-1/2}\mathcal T_iu_j,$$
which, for $\omega\in\mathcal A$, is the likelihood ratio of the $i$th private release under $f_\omega$ with respect to its law $\mu_i$ under the uniform density. Under $\mu_i$, each $\mathcal T_iu_j$ is centred: by Fubini's theorem, which applies because $|q_i(z\mid x)u_j(x)|\leq e^\alpha|u_j(x)|$,
$$
\begin{aligned}
\int\mathcal T_iu_j\,d\mu_i
&=\int_0^1u_j(x)\Bigl\{\int q_i(z\mid x)\,d\mu_i(z)\Bigr\}dx\\
&=\int_0^1u_j(x)\,dx=0,
\end{aligned}
$$
since $u_j$ has mean zero by (W4). Consequently,
$$
\begin{aligned}
\int\mathcal L_{\omega,i}\mathcal L_{\omega',i}\,d\mu_i
&=1+L^{-2r-2}\sum_{j,l=1}^L\omega_j\omega'_l\,\widetilde\lambda_j^{-1/2}\widetilde\lambda_l^{-1/2}\\
&\qquad\times\int\mathcal T_iu_j\,\mathcal T_iu_l\,d\mu_i,
\end{aligned}
$$
and, by definition of the $u_j$, $n^{-1}\sum_i\int\mathcal T_iu_j\,\mathcal T_iu_l\,d\mu_i=\lambda_j\1\{j=l\}$. Hence, by the inequality $1+y\leq e^y$, for all $\omega,\omega'\in\{-1,1\}^L$,
$$\prod_{i=1}^n\int\mathcal L_{\omega,i}\mathcal L_{\omega',i}\,d\mu_i\leq\exp\Bigl(\frac{n}{L^{2r+2}}\sum_{j=1}^L\frac{\lambda_j}{\widetilde\lambda_j}\omega_j\omega'_j\Bigr).$$
The private releases are independent, and
$$\mathcal L_n=\frac{1}{\pi(\mathcal A)}\E_\pi\Bigl\{\1_{\mathcal A}(\omega)\prod_{i=1}^n\mathcal L_{\omega,i}(Z_i)\Bigr\},$$
where $\E_\pi$ integrates over $\omega$ only. Since $\mathcal L_{\omega,i}\geq0$ for $\omega\in\mathcal A$, Fubini's theorem and the last display give
$$\E_1^{Q^n}\mathcal L_n^2\leq\frac{1}{\pi(\mathcal A)^2}\E_{\pi\otimes\pi}\exp\Bigl(\frac{n}{L^{2r+2}}\sum_{j=1}^L\frac{\lambda_j}{\widetilde\lambda_j}\omega_j\omega'_j\Bigr),$$
where the indicators of $\mathcal A$ have been dropped from a positive integrand. Put $y_j=nL^{-2r-2}\lambda_j/\widetilde\lambda_j$, so that $0\leq y_j\leq nL^{-2r-2}e^{2\bar\alpha}\alpha^2$ by Lemma~\ref{lem:spectral-consequences}. Under $\pi\otimes\pi$, the products $\omega_1\omega'_1,\ldots,\omega_L\omega'_L$ are independent Rademacher random variables. Therefore
$$
\begin{aligned}
\E_{\pi\otimes\pi}\exp\Bigl(\sum_{j=1}^Ly_j\omega_j\omega'_j\Bigr)
&=\prod_{j=1}^L\E_{\pi\otimes\pi}\exp(y_j\omega_j\omega'_j)\\
&=\prod_{j=1}^L\frac{e^{y_j}+e^{-y_j}}{2}\\
&=\prod_{j=1}^L\cosh(y_j).
\end{aligned}
$$
Since $(2k)!\geq2^kk!$ for every $k\geq0$, one has $\cosh y=\sum_{k\geq0}y^{2k}/(2k)!\leq\sum_{k\geq0}(y^2/2)^k/k!=e^{y^2/2}$, and therefore
$$
\begin{aligned}
\prod_{j=1}^L\cosh(y_j)
&\leq\exp\Bigl(\frac12\sum_{j=1}^Ly_j^2\Bigr)\\
&\leq\exp\Bigl(\frac L2\cdot\frac{n^2e^{4\bar\alpha}\alpha^4}{L^{4r+4}}\Bigr)=\exp\Bigl(\frac{n^2e^{4\bar\alpha}\alpha^4}{2L^{4r+3}}\Bigr).
\end{aligned}
$$
Combined with the previous display and $\pi(\mathcal A)\geq1-p$, this proves the lemma.
\end{proof}

\subsection{Proofs of the lower bounds}
\label{app:lower-bound-proofs}

At resolution $L$, the alternatives have size of order $L^{-r}$, and Lemma~\ref{lem:conditioned-prior-second-moment} shows that their mixture cannot be told apart from the uniform density as soon as $n^2\alpha^4L^{-4r-3}\to0$. The critical resolution is therefore $L\asymp(n\alpha^2)^{2/(4r+3)}$, at which the size of the alternatives is the private testing rate
$$\rho_n^\ast=(n\alpha^2)^{-2r/(4r+3)}.$$
To obtain the lower bound at every separation of smaller order, we take the resolution slightly finer than critical.

\begin{prop}
\label{prop:private-composite-lower}
Let $\rho_n=o(\rho_n^\ast)$. For every sequence of non-interactive $\alpha$-LDP channels and every sequence of tests $\Psi_n$,
$$
\begin{aligned}
\liminf_{n\to\infty}\Bigl[&\sup_{f\in\Sigma(s,R,U)}P_f^{Q^n}(\Psi_n=1)\\
&+\sup_{f\in H_1(\rho_n)}P_f^{Q^n}(\Psi_n=0)\Bigr]\geq1.
\end{aligned}
$$
\end{prop}

\begin{proof}
Replacing $\rho_n$ by $\rho_n\vee(\rho_n^\ast)^2$ shrinks $H_1(\rho_n)$, and therefore strengthens the claim, while preserving the assumption, because $\rho_n^\ast\to0$. We may thus assume that $\rho_n>0$. Let $L_n=2^{J_n}$ be the smallest power of two that is at least $(n\alpha^2)^{2/(4r+3)}(\rho_n^\ast/\rho_n)^{1/(2r)}$. Then $L_n\to\infty$, and
$$L_n^{-r}>2^{-r}(n\alpha^2)^{-2r/(4r+3)}\Bigl(\frac{\rho_n}{\rho_n^\ast}\Bigr)^{1/2}=2^{-r}\rho_n^\ast\Bigl(\frac{\rho_n}{\rho_n^\ast}\Bigr)^{1/2}.$$
We deduce that
$$
\begin{aligned}
\frac{\rho_n}{L_n^{-r}}&\leq2^r\Bigl(\frac{\rho_n}{\rho_n^\ast}\Bigr)^{1/2}\to0,\\
\frac{n^2\alpha^4}{L_n^{4r+3}}&\leq\Bigl(\frac{\rho_n}{\rho_n^\ast}\Bigr)^{(4r+3)/(2r)}\to0.
\end{aligned}
$$
Apply the construction of Appendix~\ref{app:prior} at level $J_n$, with the given channels, and let $\mathcal A_n=\mathcal A_{L_n,p_n}$ with $p_n=L_n^{-2}$. By Lemma~\ref{lem:probabilistic-positivity}, $\pi(\mathcal A_n)\geq1-p_n$. Moreover, the bound on $\|f_\omega-1\|_\infty$ in that lemma is of order $L_n^{-r}(\log L_n)^{1/2}$, which tends to zero, so that $f_\omega\in\Sigma(r,R,U)$ for every $\omega\in\mathcal A_n$ when $n$ is large enough. By Lemma~\ref{lem:alternative-integral-regularity}, the distance from each such $f_\omega$ to $\Sigma(s,R,U)$ is at least $L_n^{-r}/\sqrt2-RL_n^{-s}$, which exceeds $\rho_n$ for $n$ large enough, because $s>r$ and $\rho_n=o(L_n^{-r})$. The conditioned prior $\Pi_n$ of Lemma~\ref{lem:conditioned-prior-second-moment} is therefore supported on $H_1(\rho_n)$, and that lemma gives
$$\E_1^{Q^n}\mathcal L_n^2\leq(1-L_n^{-2})^{-2}\exp\Bigl(\frac{n^2\alpha^4}{2L_n^{4r+3}}\Bigr)\to1.$$
Since $\E_1^{Q^n}\mathcal L_n=1$, it follows that
$$
\begin{aligned}
\|P_{\Pi_n}^{Q^n}-P_1^{Q^n}\|_{\TV}&=\frac12\E_1^{Q^n}|\mathcal L_n-1|\\
&\leq\frac12(\E_1^{Q^n}\mathcal L_n^2-1)^{1/2}\to0.
\end{aligned}
$$
Finally, let $\Psi_n$ be any test. Applying the definition of the total variation distance to the event $\{\Psi_n=1\}$ gives
$$
\begin{aligned}
P_{\Pi_n}^{Q^n}(\Psi_n=0)&=1-P_{\Pi_n}^{Q^n}(\Psi_n=1)\\
&\geq1-P_1^{Q^n}(\Psi_n=1)\\
&\quad-\|P_{\Pi_n}^{Q^n}-P_1^{Q^n}\|_{\TV},
\end{aligned}
$$
that is, $P_1^{Q^n}(\Psi_n=1)+P_{\Pi_n}^{Q^n}(\Psi_n=0)\geq1-\|P_{\Pi_n}^{Q^n}-P_1^{Q^n}\|_{\TV}$. We now bound the two terms on the left by the two suprema in the claim. Since the uniform density belongs to $\Sigma(s,R,U)$,
$$P_1^{Q^n}(\Psi_n=1)\leq\sup_{f\in\Sigma(s,R,U)}P_f^{Q^n}(\Psi_n=1).$$
Since $P_{\Pi_n}^{Q^n}=\int P_f^{Q^n}\,d\Pi_n(f)$ and, for $n$ large enough, $\Pi_n$ is supported on $H_1(\rho_n)$,
$$
\begin{aligned}
P_{\Pi_n}^{Q^n}(\Psi_n=0)&=\int P_f^{Q^n}(\Psi_n=0)\,d\Pi_n(f)\\
&\leq\sup_{f\in H_1(\rho_n)}P_f^{Q^n}(\Psi_n=0).
\end{aligned}
$$
Hence, for $n$ large enough, the bracket in the claim is at least $1-\|P_{\Pi_n}^{Q^n}-P_1^{Q^n}\|_{\TV}$. Since the total variation distance tends to zero, the limit inferior of the bracket is at least one.
\end{proof}

\begin{proof}[Proof of Theorem~\ref{thm:testing-lower}]
Under bounded privacy, the private term of $\rho_{n,\test}(r,\alpha)$ dominates (Section~\ref{sec:model}), so that $\rho_{n,\test}(r,\alpha)\asymp(n\alpha^2)^{-2r/(4r+3)}=\rho_n^\ast$. The assumption $\rho_n=o\{\rho_{n,\test}(r,\alpha)\}$ is therefore equivalent to $\rho_n=o(\rho_n^\ast)$, and Proposition~\ref{prop:private-composite-lower} proves the theorem.
\end{proof}

We now deduce the negative side of Theorem~\ref{thm:main-adaptation}. If an honest set had a smooth-class diameter of smaller order than the private testing rate, Proposition~\ref{prop:testing-reduction} would turn it into a test whose total error stays below one at a separation where Theorem~\ref{thm:testing-lower} forbids it.

\begin{proof}[Negative side of Theorem~\ref{thm:main-adaptation}]
Suppose that $\mathcal R$ contains some $s>4r/3$ and that, for some $\gamma<1/2$, a sequence of confidence sets $C_n$ based on non-interactive $\alpha$-LDP releases is honest over $\Sigma(r,R,U)$ at level $1-\gamma$ and satisfies, for every $M_n\to\infty$,
$$\sup_{f\in\Sigma(s,R,U)}P_f^{Q^n}\{|C_n|_2>M_n\ell_n\delta_{n,\est}(s,\alpha)\}\to0.$$ Under bounded privacy $\alpha\leq \bar \alpha$, $\delta_{n,\est}(s,\alpha)\asymp(n\alpha^2)^{-s/(2s+2)}$ and $\rho_{n,\test}(r,\alpha)\asymp(n\alpha^2)^{-2r/(4r+3)}$. Since $s>4r/3$ is equivalent to $d=s/(2s+2)-2r/(4r+3)>0$,
$$\frac{\ell_n\delta_{n,\est}(s,\alpha)}{\rho_{n,\test}(r,\alpha)}\lesssim\{\log(e+n\alpha^2)\}^{\nu+1/2}(n\alpha^2)^{-d}\to0.$$
Hence the sequence
$$M_n=\Bigl\{\frac{\rho_{n,\test}(r,\alpha)}{\ell_n\delta_{n,\est}(s,\alpha)}\Bigr\}^{1/2}$$
tends to infinity, while $v_n=M_n\ell_n\delta_{n,\est}(s,\alpha)$ satisfies $v_n=o\{\rho_{n,\test}(r,\alpha)\}$. Let $\Psi_n$ be the test of Proposition~\ref{prop:testing-reduction} with threshold $v_n$. By that proposition and the assumed diameter bound, its type-I error over $\Sigma(s,R,U)$ is at most $\gamma+o(1)$, and its type-II error is at most $\gamma$ at every $f\in H_1(2v_n)$, since such $f$ lie at distance more than $v_n$ from $\Sigma(s,R,U)$. Therefore
$$
\begin{aligned}
\limsup_{n\to\infty}\Bigl[&\sup_{f\in\Sigma(s,R,U)}P_f^{Q^n}(\Psi_n=1)\\
&+\sup_{f\in H_1(2v_n)}P_f^{Q^n}(\Psi_n=0)\Bigr]\leq2\gamma<1,
\end{aligned}
$$
which contradicts Theorem~\ref{thm:testing-lower}, applied with $\rho_n=2v_n=o\{\rho_{n,\test}(r,\alpha)\}$.
\end{proof}

\section{Adaptation under separation}
\label{app:separated-adaptation}

This appendix proves Theorem~\ref{thm:density-separated}. Beyond the adaptation boundary, adaptation becomes possible once a critical region of rough densities close to the smooth class is removed \citep{bullAdaptiveConfidenceSets2012}; see also \citet{gineMathematicalFoundationsInfiniteDimensional2016}. The construction combines an adaptive estimator with a test of the smooth class against the remaining rough densities, and uses the outcome of the test to choose the radius. \citet{carpentierTestingRegularitySmooth2015} constructs such tests of regularity in the Gaussian sequence model. We build the test from the private projected-loss statistic of Appendix~\ref{app:projected-loss-confidence}. The separation it requires is of order $\ell_n^2\{\delta_{n,\est}(s,\alpha)\vee\rho_{n,\test}(r,\alpha)\}$. When $s>4r/3$, this is $\ell_n^2\rho_{n,\test}(r,\alpha)$, which, under bounded privacy, matches the lower bound of Theorem~\ref{thm:testing-lower} up to the factor $\ell_n^2$. From Appendix~\ref{app:testing-lower-impossibility}, the removed region can therefore not be made of substantially smaller order.

We use three independent sample splits: the first produces a reference estimator, the second carries out the test, and the third provides the centre of the confidence set. The third split is not essential, since the coverage argument below only uses a union bound and the centre could equally be computed on the first split, We keep it nonetheless so that the centre does not depend on the data used by the test. Throughout this appendix, $r<s$ are two regularities in $(0,S)$, and $H_1(\rho)$ is the set of alternatives of Section~\ref{sec:smoothness-testing}.

\subsection{Construction}
\label{app:separated-construction}

Partition the sample deterministically into groups of sizes $m_1=m_2=\lfloor n/3\rfloor$ and $m_3=n-m_1-m_2\geq m_1$. Put
$$\gamma_n=\min\{\gamma/2,\ell_n^{-2}\},$$
the level allotted to each of the two errors of the test. Since $\gamma_n\to0$, the test eventually selects the smooth confidence set with probability tending to one under $\Sigma(s,R,U)$.

On the first split, compute the adaptive estimator $\widetilde f_{m_1}$ of Proposition~\ref{prop:adaptive-private-estimator} and its $L^2$ projection $\overline f_s$ onto $\Sigma(s,R,U)$, which is well defined by Lemma~\ref{lem:smooth-private-estimator} below. The test then compares $f$ with $\overline f_s$ through the projected-loss statistic $\widehat T_J=U_{m_2}(\overline f_s,J)$ of the second split, for which $v_J$, $w_J$, $\mathcal H_{\eta,J}$ and $F_\eta(J)$ are computed with $m_2$. For $J>j_0$, put
$$\tau_J=\frac{1}{\gamma_n}\Bigl\{a_{m_1}(s)+\frac12\sqrt{4v_Ja_{m_1}(s)+w_J^2}\Bigr\},$$
let $\widehat\Psi_J=\1\{\widehat T_J>\tau_J\}$, and put
$$b_{r,s}(J)=b_r(J)+b_s(J)+\frac{2R^22^{-J(r+s)}}{1-2^{-(r+s)}}.$$
The threshold $\tau_J$ controls the type-I error through Markov's inequality, and $b_{r,s}(J)$ bounds the part of the squared separation that is lost by projecting onto the levels below $J$ (Lemma~\ref{lem:projected-density-separation}). Let
$$J_{\sep}=\min\operatorname*{arg\,min}_{J>j_0}\{b_{r,s}(J)+\mathcal H_{\gamma_n,J}(\tau_J)\},$$
which exists because $\mathcal H_{\gamma_n,J}(\tau_J)\geq F_{\gamma_n}(J)$ tends to infinity with $J$, as in the proof of Lemma~\ref{lem:finite-resolution-rates}, and let
$$\rho_{\mathrm{FS}}^2=b_{r,s}(J_{\sep})+\mathcal H_{\gamma_n,J_{\sep}}(\tau_{J_{\sep}})$$
be the square of the separation that the test can guarantee (Proposition~\ref{prop:separated-test} below). The second split releases $Z_{i,J_{\sep}}$, and the test is $\widehat\Psi_{J_{\sep}}$.

On the third split, compute $\widetilde f_{m_3}$ and, for $t\in\{r,s\}$,
$$C_t=\Bigl\{h\in\Sigma(t,R,U):\|h-\widetilde f_{m_3}\|_2^2\leq\frac{2a_{m_3}(t)}{\gamma}\Bigr\}.$$
Finally, if $\rho_n<\rho_{\mathrm{FS}}$, set $C_n^{\den}(\gamma)=\Sigma(r,R,U)$. Otherwise, set $C_n^{\den}(\gamma)=C_s$ if $\widehat\Psi_{J_{\sep}}=0$ and $C_n^{\den}(\gamma)=C_r$ if $\widehat\Psi_{J_{\sep}}=1$. The resolution $J_{\sep}$, the threshold $\tau_{J_{\sep}}$ and the radii of $C_r$ and $C_s$ depend only on $n$, $\alpha$, $\gamma$ and the fixed parameters of the model.

\subsection{Finite-sample bounds}
\label{app:separated-bounds}

The test only sees the levels below $J$, whereas the separation is measured by the full $L^2$ distance. The next lemma bounds the loss incurred by this truncation.

\begin{lem}
\label{lem:projected-density-separation}
For every $f\in\Sigma(r,R,U)$ and every $J>j_0$,
$$
\begin{aligned}
&\inf_{g\in\Sigma(s,R,U)}\|K_J(f-g)\|_2^2\\
&\quad\geq\dist_2(f,\Sigma(s,R,U))^2-b_{r,s}(J).
\end{aligned}
$$
\end{lem}

\begin{proof}
Let $g\in\Sigma(s,R,U)$. For $j\geq J$, the Besov constraints give $\|\beta_j(f)-\beta_j(g)\|_{\ell^2}\leq R(2^{-jr}+2^{-js})$, so that, by (W1),
$$\|(I-K_J)(f-g)\|_2^2\leq R^2\sum_{j=J}^\infty(2^{-jr}+2^{-js})^2=b_{r,s}(J).$$
By orthogonality, $\|K_J(f-g)\|_2^2=\|f-g\|_2^2-\|(I-K_J)(f-g)\|_2^2\geq\dist_2(f,\Sigma(s,R,U))^2-b_{r,s}(J)$, and the claim follows by taking the infimum over $g$.
\end{proof}

The reference $\overline f_s$ must lie in the smooth class, so we project $\widetilde f_{m_1}$ on it. Under the null, the projection then inherits the risk of $\widetilde f_{m_1}$, and under a separated alternative, its distance to $f$ is at least the distance from $f$ to the smooth class.

\begin{lem}
\label{lem:smooth-private-estimator}
The smooth class is a non-empty, closed and convex subset of $L^2([0,1])$. Consequently, the $L^2$ projection $\overline f_s$ of $\widetilde f_{m_1}$ onto $\Sigma(s,R,U)$ is well defined, it is a measurable function of the first-split release, and
$$\sup_{f\in\Sigma(s,R,U)}\E_f^{Q^{m_1}}\|\overline f_s-f\|_2^2\leq a_{m_1}(s).$$
\end{lem}

\begin{proof}
The class contains the uniform density, and it is convex because the constraints $f\geq0$, $\int_0^1f=1$, $\|f\|_\infty\leq U$ and $\|f\|_{B^s_{2,\infty}}\leq R$ are preserved by convex combinations.

Let $f$ be the $L^2$ limit of a sequence in $\Sigma(s,R,U)$, so that the sequence also converges in $L^1([0,1])$ and $\int_0^1f=1$. Moreover, every sequence converging in $L^2$ has a subsequence converging almost everywhere. Since every element $h$ of the sequence satisfies $0\leq h\leq U$ almost everywhere and the interval $[0,U]$ is closed, the almost-everywhere limit satisfies $0\leq f\leq U$. Next, each wavelet coefficient is a continuous linear functional on $L^2([0,1])$: by the Cauchy-Schwarz inequality and (W1), $|\langle h-f,\psi_{jk}\rangle|\leq\|h-f\|_2$, so every coefficient converges along the sequence. Since each level contains finitely many coefficients, the norms $\|\beta_j(h)\|_{\ell^2}$ converge to $\|\beta_j(f)\|_{\ell^2}$, and the bound $2^{js}\|\beta_j(h)\|_{\ell^2}\leq R$, valid for every element $h$ and every $j\geq j_0$, therefore also holds for $f$. We conclude that $\|f\|_{B^s_{2,\infty}}\leq R$. The class is therefore closed.

The projection onto a non-empty closed convex subset of a Hilbert space exists, is unique and is non-expansive, hence continuous, so that $\overline f_s$ is a measurable function of $\widetilde f_{m_1}$. For $f\in\Sigma(s,R,U)$, non-expansiveness gives $\|\overline f_s-f\|_2\leq\|\widetilde f_{m_1}-f\|_2$, and Proposition~\ref{prop:adaptive-private-estimator} completes the proof.
\end{proof}

The next proposition controls both errors of the test at a fixed resolution. Under the smooth null, $\overline f_s$ is close to $f$, so the statistic is small. Under a separated alternative, Lemma~\ref{lem:projected-density-separation} ensures that enough of the separation remains below level $J$ for the statistic to be large.

\begin{prop}
\label{prop:separated-test}
Let $J>j_0$. Then
$$\sup_{f\in\Sigma(s,R,U)}P_f^{Q^{m_1+m_2}}(\widehat\Psi_J=1)\leq\gamma_n,$$
and, if $\rho^2\geq b_{r,s}(J)+\mathcal H_{\gamma_n,J}(\tau_J)$,
$$\sup_{f\in H_1(\rho)}P_f^{Q^{m_1+m_2}}(\widehat\Psi_J=0)\leq\gamma_n.$$
\end{prop}

\begin{proof}
Given the first split, $\overline f_s$ is fixed, and Lemma~\ref{lem:projected-loss-moments}, applied with $g=\overline f_s$ and $m=m_2$, gives $\E_2\widehat T_J=q_J(\overline f_s)^2$,
$$\Var_2(\widehat T_J)\leq4v_Jq_J(\overline f_s)^2+w_J^2,$$
and $\E_2\widehat T_J^+\leq q_J(\overline f_s)^2+\frac12\{4v_Jq_J(\overline f_s)^2+w_J^2\}^{1/2}$, where $q_J(\overline f_s)^2=\|K_J(f-\overline f_s)\|_2^2$.

Let $f\in\Sigma(s,R,U)$. The squared projected loss $q_J(\overline f_s)^2$ depends only on the first split, and, since $K_J$ is a contraction, Lemma~\ref{lem:smooth-private-estimator} gives
$$\E_f^{Q^{m_1}}q_J(\overline f_s)^2\leq\E_f^{Q^{m_1}}\|f-\overline f_s\|_2^2\leq a_{m_1}(s).$$
Taking expectations over the first split in the bound on $\E_2\widehat T_J^+$ gives
$$
\begin{aligned}
\E_f^{Q^{m_1+m_2}}\widehat T_J^+
&\leq\E_f^{Q^{m_1}}q_J(\overline f_s)^2\\
&\quad+\frac12\E_f^{Q^{m_1}}\{4v_Jq_J(\overline f_s)^2+w_J^2\}^{1/2}.
\end{aligned}
$$
The map $y\mapsto(4v_Jy+w_J^2)^{1/2}$ is concave and non-decreasing on $[0,\infty)$, so Jensen's inequality and the previous bound give
$$
\begin{aligned}
&\E_f^{Q^{m_1}}\{4v_Jq_J(\overline f_s)^2+w_J^2\}^{1/2}\\
&\quad\leq\{4v_J\E_f^{Q^{m_1}}q_J(\overline f_s)^2+w_J^2\}^{1/2}\\
&\quad\leq\{4v_Ja_{m_1}(s)+w_J^2\}^{1/2}.
\end{aligned}
$$
Hence, by the definition of $\tau_J$,
$$
\begin{aligned}
\E_f^{Q^{m_1+m_2}}\widehat T_J^+
&\leq a_{m_1}(s)+\frac12\{4v_Ja_{m_1}(s)+w_J^2\}^{1/2}\\
&=\gamma_n\tau_J.
\end{aligned}
$$
As $\tau_J>0$, the events $\{\widehat T_J>\tau_J\}$ and $\{\widehat T_J^+>\tau_J\}$ coincide, and Markov's inequality gives
$$
\begin{aligned}
P_f^{Q^{m_1+m_2}}(\widehat\Psi_J=1)
&=P_f^{Q^{m_1+m_2}}(\widehat T_J^+>\tau_J)\\
&\leq\frac{\E_f^{Q^{m_1+m_2}}\widehat T_J^+}{\tau_J}\leq\gamma_n.
\end{aligned}
$$

Now let $f\in H_1(\rho)$. Since $\overline f_s\in\Sigma(s,R,U)$, Lemma~\ref{lem:projected-density-separation} gives
$$q_J(\overline f_s)^2\geq\rho^2-b_{r,s}(J)\geq\mathcal H_{\gamma_n,J}(\tau_J)>\tau_J.$$
The map $y\mapsto(y-\tau_J)^2-c_{\gamma_n}(4v_Jy+w_J^2)$ is a convex quadratic whose larger root is $\mathcal H_{\gamma_n,J}(\tau_J)$, as computed in the proof of Proposition~\ref{prop:projected-loss-upper}. It is therefore non-negative at $y=q_J(\overline f_s)^2$, which gives
$$\{q_J(\overline f_s)^2-\tau_J\}^2\geq c_{\gamma_n}\{4v_Jq_J(\overline f_s)^2+w_J^2\}.$$
Cantelli's applied, given the first split, to $\widehat T_J$ whose conditional mean is $q_J(\overline f_s)^2$ and whose conditional variance is at most $4v_Jq_J(\overline f_s)^2+w_J^2$, and the distance $q_J(\overline f_s)^2-\tau_J>0$ gives
$$
\begin{aligned}
P_2(\widehat T_J\leq\tau_J)
&\leq\frac{4v_Jq_J(\overline f_s)^2+w_J^2}{4v_Jq_J(\overline f_s)^2+w_J^2+\{q_J(\overline f_s)^2-\tau_J\}^2}\\
&\leq\frac{1}{1+c_{\gamma_n}}=\gamma_n,
\end{aligned}
$$
where we used that $\{\widehat T_J\leq\tau_J\}$ is exactly the event that $\widehat T_J$ lies at a distance below $q_J(\overline f_s)^2-\tau_J$ from its expectation, and the second inequality uses the lower bound on $\{q_J(\overline f_s)^2-\tau_J\}^2$ obtained above and $1+c_{\gamma_n}=1/\gamma_n$. Averaging over the first split completes the proof.
\end{proof}

\subsection{Proof of Theorem~\ref{thm:density-separated}}
\label{app:separated-proof}

\begin{proof}
We prove honesty at level $1-\gamma$, and the two diameter bounds in turn.

For honesty, note first that both classes in the claim are contained in $\Sigma(r,R,U)$, so there is nothing to prove when $C_n^{\den}(\gamma)=\Sigma(r,R,U)$. We may therefore assume that $\rho_n\geq\rho_{\mathrm{FS}}$. For $t\in\{r,s\}$ and $f\in\Sigma(t,R,U)$, one has $f\in C_t$ unless $\|\widetilde f_{m_3}-f\|_2^2>2a_{m_3}(t)/\gamma$, an event of probability at most $\gamma/2$ by Markov's inequality and Proposition~\ref{prop:adaptive-private-estimator}. For $f\in\Sigma(s,R,U)$, non-coverage requires $f\notin C_s$ or $\widehat\Psi_{J_{\sep}}=1$, while for $f\in H_1(\rho_n)$, it requires $f\notin C_r$ or $\widehat\Psi_{J_{\sep}}=0$. Since $\rho_n^2\geq\rho_{\mathrm{FS}}^2$, Proposition~\ref{prop:separated-test} bounds both testing errors by $\gamma_n\leq\gamma/2$, and a union bound gives $P_f^{Q^n}\{f\notin C_n^{\den}(\gamma)\}\leq\gamma$ in both cases, for every $n$.

We next show that $\rho_n\geq\rho_{\mathrm{FS}}$ for all large $n$. By Proposition~\ref{prop:projected-loss-upper}, $\mathcal H_{\gamma_n,J}(\tau_J)\leq2\tau_J+F_{\gamma_n}(J)$. Since $\{4v_Ja_{m_1}(s)+w_J^2\}^{1/2}\leq a_{m_1}(s)+v_J+w_J$, one has $\tau_J\leq\frac32\gamma_n^{-1}\{a_{m_1}(s)+v_J+w_J\}$, and, as in the proof of Lemma~\ref{lem:finite-resolution-rates}, $F_{\gamma_n}(J)\leq4c_{\gamma_n}v_J+\sqrt{c_{\gamma_n}}w_J\leq4\gamma_n^{-1}(v_J+w_J)$, because $c_{\gamma_n}\leq\gamma_n^{-1}$ and $\gamma_n\leq1$. Moreover, since $s>r$, one has $b_s(J)\leq b_r(J)$ and the last term of $b_{r,s}(J)$ is at most $2b_r(J)$, so that $b_{r,s}(J)\leq4b_r(J)$. Evaluating the criterion defining $J_{\sep}$ at every $J$ and using Lemma~\ref{lem:finite-resolution-rates}, we obtain
$$
\begin{aligned}
\rho_{\mathrm{FS}}^2
&\leq\frac{7}{\gamma_n}\Bigl[a_{m_1}(s)+\inf_{J>j_0}\{b_r(J)+v_J+w_J\}\Bigr]\\
&\lesssim\frac{1}{\gamma_n}\Bigl[a_{m_1}(s)+\inf_{J>j_0}\{b_r(J)+\Gamma_{m_2,\alpha}(J)\}\Bigr].
\end{aligned}
$$
By Proposition~\ref{prop:adaptive-private-estimator} and Lemma~\ref{lem:finite-resolution-rates}, the bracket is at most a constant multiple of $\ell_{m_1}^2\delta_{m_1,\est}(s,\alpha)^2+\ell_{m_2}^2\rho_{m_2,\test}(r,\alpha)^2$. Since $n/5\leq m_1=m_2\leq m_3\leq n$ and $\ell_m$ increases with $m$, this is at most a constant multiple of $\ell_n^2\{\delta_{n,\est}(s,\alpha)\vee\rho_{n,\test}(r,\alpha)\}^2$. As $\gamma_n^{-1}\leq2\ell_n^2/\gamma$, it follows that
$$\rho_{\mathrm{FS}}\lesssim_\gamma\ell_n^2\{\delta_{n,\est}(s,\alpha)\vee\rho_{n,\test}(r,\alpha)\},$$
so that $\rho_n\geq\rho_{\mathrm{FS}}$ for all large $n$ when the constant $M$ of Theorem~\ref{thm:density-separated} exceeds the implicit constant.

It remains to bound the diameters, for $n$ large enough that $\rho_n\geq\rho_{\mathrm{FS}}$. Let $f\in\Sigma(s,R,U)$. By Proposition~\ref{prop:separated-test}, $C_n^{\den}(\gamma)=C_s$ with probability at least $1-\gamma_n$, which tends to one because $\ell_n\to\infty$, and, by Proposition~\ref{prop:adaptive-private-estimator},
$$|C_s|_2\leq2\Bigl\{\frac{2a_{m_3}(s)}{\gamma}\Bigr\}^{1/2}\lesssim_\gamma\ell_n\delta_{n,\est}(s,\alpha),$$
which proves the bound over $\Sigma(s,R,U)$. For the bound over $H_1(\rho_n)$, note that $a_m(s)\leq a_m(r)$, because the only terms of $A_m(t)$ that depend on $t$, namely $b_t(K_m)$ and the minima in which $R^22^{-2jt}$ appears, are non-increasing in $t$. Hence both $C_r$ and $C_s$ have diameter at most $2\{2a_{m_3}(r)/\gamma\}^{1/2}\lesssim_\gamma\ell_n\delta_{n,\est}(r,\alpha)$, whatever the outcome of the test, which completes the proof.
\end{proof}

\section{Numerical implementation}
\label{app:numerical-implementation}

We implement the two-split construction of Section \ref{sec:construction} on a discretised finite grid. The numerical basis is periodic Daubechies with $S=8$. A nested Gram correction makes the grid coordinates orthonormal before we calculate the estimator, covariance bound and confidence radius.

\subsection{Finite-grid model and densities}

The simulations approximate the continuous density problem on a fine regular grid over $[0,1]$. We divide the interval into $M=4096=2^{12}$ cells with midpoints
$$
x_c=\frac{c-1/2}{M},\qquad c=1,\ldots,M.
$$
A density is represented by its values $f=(f_1,\ldots,f_M)$ on these cells, subject to $0\leq f_c=f(x_c)\leq U$ and
$$
\frac{1}{M}\sum_{c=1}^M f_c=1.
$$
Sampling from this discrete approximation means selecting cell $c$ with probability $f_c/M$. For any grid function $h=(h_1,\ldots,h_M)$, we use
$$
\|h\|_M^2=\frac{1}{M}\sum_{c=1}^M h_c^2,
$$
which is exactly the squared $L^2$ norm of the corresponding cellwise-constant function on $[0,1]$. Thus the finite-grid implementation preserves the $L^2$--loss used in the theoretical construction.

We next construct a discrete analogue of the orthonormal wavelet coordinates used in the theory. Let $B_J$ contain the periodic Daubechies scaling coordinate (corresponding to $j=-1$) together with wavelet coordinates below resolution $J$, evaluated on the same grid. Because evaluation on a finite grid does not preserve exact orthonormality, we correct the resulting Gram matrix
$$
G_J=\frac{1}{M}B_J^\top B_J=L_J^\top L_J,
$$
where $L_J$ is its upper-triangular Cholesky factor. We define
\begin{equation}\label{eq:gram_cor_matrix}
W_J=L_J^{-\top},
\qquad
\Phi_J=B_JL_J^{-1},
\end{equation}
so that
$$
\frac{1}{M}\Phi_J^\top\Phi_J=I.
$$
The columns of $\Phi_J$ therefore form the orthonormal coordinates used throughout the simulation. These coordinates remain nested as $J$ increases, matching the multiresolution structure of the continuous construction.

For a grid density $f$, its coefficients through resolution $J$ are
$$
\theta_J(f)=\frac{1}{M}\Phi_J^\top f.
$$
We write $\theta_{\mathrm{sc}}(f)$ for the scaling coefficient (for $j=-1$) and $\theta_j(f)$ for the detail coefficients at level $-1<j<J$. 
The finite-grid class $\Sigma_M(t,R,U)$ consists of density vectors satisfying $0\leq f_c\leq U$, $M^{-1}\sum_c f_c=1$, and
\[
\max\left\{
|\theta_{\mathrm{sc}}(f)|,\,
\max_{0\leq j<12}2^{jt}\|\theta_j(f)\|_2
\right\}\leq R.
\]
For resolutions below the full grid, we retain the same geometric tail bound as in the theoretical construction,
$$
b_t(J)=\frac{R^2 2^{-2Jt}}{1-2^{-2t}}.
$$

We use three fixed generating densities chosen to have different shapes while remaining inside the same honesty classes. The sine density is proportional to
$$
1+0.5\sin(2\pi x_c),
$$
and the triangular density is proportional to
$$
0.5+2\min(x_c,1-x_c).
$$
For the third density, we first generate a discrete Brownian path by taking cumulative sums of $M$ independent $N(0,1/M)$ increments and centring the path by its grid mean. If the resulting values are denoted by $w_c$, the exponentiated-Brownian density is proportional to
$$
\exp(0.8w_c).
$$
Each density is normalised so that its grid mean equals one, and the same fixed exponentiated-Brownian realisation is used throughout the study.

All three generating densities belong to the baseline class $\Sigma_M(0.5,1,3)$. At $t=0.5$, their largest weighted detail norms are $0.360239$, $0.273466$ and $0.545235$ for the sine, triangular and exponentiated-Brownian densities, respectively. Their scaling coefficients equal one and their maximum heights are below $3$. Since the classes become larger as the regularity index decreases, all three densities also belong to $\Sigma_M(r,1,3)$ for every $r\in\{0.2,0.3,0.4,0.5\}$ considered in the honesty-class experiment. Varying $r$ therefore changes only the class over which coverage is required, while leaving the generating density fixed.

\subsection{Private release and centre}

We use the same two-split construction as in Section~\ref{sec:construction}. The first split, of size $m_1=\lfloor n/2\rfloor$, is used to construct the adaptive centre, while the remaining $m_2=n-m_1$ observations are reserved for estimating its projected loss.

We first describe the private wavelet release. For a detail coordinate at level $j$, define
$$
d_{j}=(j\vee1)^\nu2^{j/2},
$$
with $d_{\mathrm{sc}}=d_{-1}=1$ for the scaling coordinate. If coefficients are released through level $L-1$, their weighted sensitivity on the finite grid is
$$\Delta_L=2\max_{1\leq c\leq M}\sum_{a:\,0\leq j(a)<L}\frac{|(\Phi_L)_{ca}|}{d_{j(a)}}.$$
The constant scaling coordinate does not contribute to sensitivity.
Individual $i$, whose observation falls in cell $X_i$, releases
$$Y_{i,a}=(\Phi_L)_{X_i,a}+\sigma_{j(a)}\xi_{i,a},\qquad \sigma_{j(a)}=\frac{\Delta_Ld_{j(a)}}{\alpha},$$
where the $\xi_{i,a}$ are independent standard Laplace variables. This is the finite-grid version of the level-dependent release in Appendix~\ref{sec: privacy release}. The sensitivity bound ensures that the resulting channel is $\alpha$-LDP. The deterministic Gram correction from the previous subsection is applied before release, so that independent privacy noise is added directly in the orthonormal coordinates $\Phi_L$.

On the first split, the finest released resolution is chosen from the effective private sample size
$$
N_1=m_1\alpha^2,
\qquad
K=\max\left\{1,\left\lceil\frac12\log_2(e+N_1)\right\rceil\right\}.
$$
This is the finite-grid analogue of the resolution used by the adaptive estimator in Appendix~\ref{app: adaptive_estimator}. We average the releases to estimate the wavelet coefficients and threshold the detail coefficients to remove components that cannot be distinguished reliably from sampling and privacy noise.

To define the threshold, let
$$s_a^2=\sigma_{j(a)}^2,\qquad \lambda_a=\sigma_{j(a)},$$
so that $2s_a^2$ is the variance of the privacy noise in coordinate $a$. With
$$
\varepsilon=(e+N_1)^{-2},
\qquad
x=\log(2^{K+1}/\varepsilon),
$$
we use
$$
\begin{aligned}
\tau_a=2\Bigg\{
&\sqrt{\frac{2(U+2s_a^2)x}{m_1}}\\
&+\max\left(
\frac{2\|\Phi_{K,a}\|_\infty}{3},
\lambda_a
\right)\frac{x}{m_1}
\Bigg\}.
\end{aligned}
$$
The scaling coefficient is retained, while a detail coefficient is retained only when the absolute value of its empirical mean exceeds $\tau_a$. Finally, as in the theoretical construction, we project the resulting estimator onto the $L^2$ ball of radius $\sqrt U$. We denote this first-split estimator by $\widetilde f_{m_1}$.

For comparison, the non-private centre is constructed from the same first-split observations without Laplace noise. We use
$$K=\max\left\{1,\left\lceil\frac12\log_2(e+m_1)\right\rceil\right\},$$
set the privacy-noise terms in the threshold to zero, and otherwise apply the same thresholding and $L^2$ projection. This keeps the private and non-private constructions as comparable as possible while allowing each to use an appropriate resolution.

\subsection{Projection and confidence radius}

The second split is used to estimate how far the generating density lies from the first-split centre at a chosen wavelet resolution. This is the finite-grid analogue of the projected-loss construction in Section~\ref{sec:construction}. The main additional ingredient under privacy is that the variance of this estimate must include the covariance introduced by the Laplace noise.

For a projection through level $J-1$, let $e_J=M^{-1}\Phi_J^\top\mathbf 1$ and $G_J^\Phi=M^{-1}\Phi_J^\top\Phi_J$. We bound the covariance of a single private release by
$$\mathsf V_J=U(G_J^\Phi-e_Je_J^\top)+2\operatorname{diag}(\sigma_a^2),$$
where the first term accounts for ordinary sampling variation and the second for privacy noise. The privacy contribution is diagonal because the noise is independent in these coordinates. As in the theoretical construction, define
$$v_J=\frac{\|\mathsf V_J\|_{\mathrm{op}}}{m_2},\qquad w_J=\left\{\frac{2\|\mathsf V_J\|_{\mathrm{F}}^2}{m_2(m_2-1)}\right\}^{1/2}.$$
These quantities control the part of the projected-loss variance that depends on the unknown loss and the part that remains even when the projected loss is zero.

We convert this variance bound into an upper confidence bound using the same function as in Proposition~\ref{prop:projected-loss-upper}. With $c_\gamma=(1-\gamma)/\gamma$, set
$$\mathcal H_{\gamma,J}(u)=u+2c_\gamma v_J+\sqrt{4c_\gamma v_Ju+4c_\gamma^2v_J^2+c_\gamma w_J^2},$$
and write $F_\gamma(J)=\mathcal H_{\gamma,J}(0)$.

The projection resolution balances approximation error against uncertainty in estimating the projected loss. We therefore choose
$$J^\star=\min\operatorname*{arg\,min}_{J\in\{1,\ldots,12\}}\{b_r(J)+F_\gamma(J)\}.$$
$J^\star$ implements on the finite grid the same bias-uncertainty trade-off used in the theoretical construction. Ties are resolved in favour of the smaller resolution.

To ensure that the privacy mechanism is fixed before observing the data, the second split releases coefficients through a deterministic resolution level $E-1$ large enough to contain every projection required by the simulation design. The associated privacy scales are $\sigma_{j}=\Delta_Ed_j/\alpha$. We choose the smallest such envelope that contains the resolutions selected by the full procedure and the full ablation over the privacy levels and honesty classes considered. This envelope affects only which coefficients are released. Each reported confidence set is still constructed at its own selected resolution $J^\star$.

Let $Y_{i,J}$ denote the second-split release restricted to the coordinates below level $J$, and let $\eta_J=\theta_J(\widetilde f_{m_1})$ be the corresponding coefficients of the first-split estimator. We estimate the squared projected loss using the diagonal-free statistic
$$U_J=\frac{1}{m_2(m_2-1)}\sum_{i\neq i'}\left\langle Y_{i,J}-\eta_J,Y_{i',J}-\eta_J\right\rangle.$$
Conditional on the first split, $\E(U_J)=\|\theta_J(f)-\eta_J\|_2^2$, and its variance is bounded by
$$4v_J\|\theta_J(f)-\eta_J\|_2^2+w_J^2.$$
Thus $U_J$ is the finite-grid counterpart of the projected-loss statistic analysed in Appendix~\ref{app:projected-loss-confidence}.

At the selected resolution, the centre is $g_n=\Phi_{J^\star}\eta_{J^\star}$ and the squared radius is
$$R_n^2=b_r(J^\star)+\mathcal H_{\gamma,J^\star}(U_{J^\star}^+).$$
The resulting confidence set is
$$C_{n,M}(\gamma)=\{h\in\Sigma_M(r,R,U):\|h-g_n\|_M^2\leq R_n^2\}.$$
Coverage is evaluated by checking whether the generating density satisfies this inequality. For diameter comparisons, we report $D=2R_n$, the diameter of the enclosing ball, so that $|C_{n,M}(\gamma)|_M\leq D$.

The two coverage ablations modify only how the privacy contribution to $\mathsf V_J$ is used. The radius ablation selects the same $J^\star$ as the full procedure but removes the privacy-noise term from $\mathsf V_J$ when constructing the radius. The full ablation removes this term both when selecting $J^\star$ and when constructing the radius. In both cases the observations themselves remain privatised, so the comparison isolates the effect of accounting for privacy-induced uncertainty rather than the effect of adding privacy noise. At $\alpha=\infty$, the privacy term vanishes and all three procedures reduce to the same non-private construction.

\subsection{Simulation design and computation}

We evaluate the construction across sample size, privacy level and honesty class. We use $n\in\{1000,5000,20000\}$ and $5000$ Monte Carlo replications for each configuration. 
The private experiments use $23$ privacy levels, $\alpha\in\mathcal A$, where
\[
\begin{aligned}
\mathcal A=\{&
0.1,0.15,0.2,0.25,0.3,0.35,\\
&0.4,0.5,0.625,0.75,0.875,1,\\
&1.25,1.5,1.75,2,2.25,2.5,\\
&3,3.5,4,4.5,5
\}.
\end{aligned}
\]
We also include the non-private construction, denoted by $\alpha=\infty$.
These values cover regimes from strong privacy, where the added noise dominates the confidence radius, to weak privacy, where the construction approaches its non-private counterpart.

For the coverage experiment, we fix the baseline honesty class at $r=0.5$ and compare the full construction with the two covariance ablations described above. For the diameter experiment, we use the full construction and vary $r\in\{0.2,0.3,0.4,0.5\}$. Because the generating density is held fixed across these values of $r$, this comparison isolates the effect of enlarging the class over which coverage is required.

Within each replication, we generate the two raw sample splits once and reuse them across all procedures being compared. Private procedures also use common underlying Laplace draws, rescaled according to the privacy level $\alpha$. Thus comparisons across privacy levels, honesty classes and ablations are paired within a replication. This coupling reduces Monte Carlo variation in their differences while leaving the marginal distribution, and hence the privacy guarantee, of each individual procedure unchanged.

For coverage, we report the empirical proportion of replications in which the generating density belongs to the confidence set, together with pointwise $95\%$ Clopper-Pearson intervals. For diameter, we report the Monte Carlo mean $\overline D$ of the enclosing-ball diameter $D=2R_n$. If $s_D$ denotes its sample standard deviation across the $5000$ replications, the Monte Carlo standard error is $s_D/\sqrt{5000}$. The plotted error bars are $\overline D\pm1.96s_D/\sqrt{5000}$. 

\subsection{Diameter comparisons}

Table~\ref{tab:diameter-comparison} shows the mean enclosing-ball diameters for all three generating densities at selected privacy levels. It complements Figure~\ref{fig:honesty-cost}, which shows the same comparison for the sine density across the full privacy range. We use the table to check whether the effect of the honesty class persists across different density shapes.

The same qualitative pattern appears for all three densities. Under strong privacy, the diameters are dominated by the privacy contribution and therefore change little with either the generating density or the honesty class. As privacy weakens, the effect of the honesty requirement becomes more visible. At $n=20000$ and $\alpha=5$, decreasing $r$ from $0.5$ to $0.2$ increases the mean diameter by between $77\%$ and $86\%$ across the three densities. Without privacy, the corresponding increase is between $47\%$ and $58\%$.

Differences between the generating densities also become more apparent as privacy weakens. For example, at $n=5000$, $r=0.5$ and $\alpha=\infty$, the mean diameter ranges from $1.048$ for the triangular density to $1.263$ for the exponentiated-Brownian density. At $n=20000$, the ordering changes, with the exponentiated-Brownian density giving the smallest non-private mean. These differences arise from the projected loss for the particular generating density, whereas the systematic increase as $r$ decreases reflects the larger class over which coverage must be guaranteed.

\begin{table*}[p]
\centering
\caption{\emph{Mean enclosing-ball diameters across the three generating densities.} Entries are means of $D=2R_n$ over $5000$ replications. Monte Carlo standard errors were computed for every entry and are below $0.15\%$ of the corresponding mean throughout; they are omitted for readability. Columns vary the honesty class while holding the generating density fixed. The table reports selected privacy levels; $\alpha=\infty$ denotes the non-private construction.}
\label{tab:diameter-comparison}
\footnotesize
\setlength{\tabcolsep}{4pt}
\renewcommand{\arraystretch}{1.02}
\begin{tabular}{llrrrr}
\toprule
Density & $\alpha$ & $r=0.2$ & $r=0.3$ & $r=0.4$ & $r=0.5$ \\
\midrule
\multicolumn{6}{c}{$n=1000$} \\
\midrule
Sine & $0.1$ & 93.1966 & 93.1710 & 93.1584 & 93.1509 \\
 & $0.5$ & 19.1000 & 18.9748 & 18.9125 & 18.8755 \\
 & $1$ & 10.1304 & 9.8924 & 9.7723 & 9.7005 \\
 & $5$ & 4.2776 & 3.6789 & 3.3428 & 3.1267 \\
 & $\infty$ & 2.4303 & 2.0363 & 1.8851 & 1.7933 \\
\addlinespace[2pt]
Triangular & $0.1$ & 93.2024 & 93.1769 & 93.1642 & 93.1567 \\
 & $0.5$ & 19.1001 & 18.9749 & 18.9126 & 18.8756 \\
 & $1$ & 10.1305 & 9.8925 & 9.7724 & 9.7006 \\
 & $5$ & 4.2655 & 3.6648 & 3.3274 & 3.1102 \\
 & $\infty$ & 2.3814 & 1.9732 & 1.8130 & 1.7174 \\
\addlinespace[2pt]
Exponentiated Brownian & $0.1$ & 93.1797 & 93.1541 & 93.1415 & 93.1340 \\
 & $0.5$ & 19.0929 & 18.9677 & 18.9053 & 18.8683 \\
 & $1$ & 10.1158 & 9.8775 & 9.7572 & 9.6853 \\
 & $5$ & 4.2354 & 3.6298 & 3.2888 & 3.0690 \\
 & $\infty$ & 2.4903 & 2.1088 & 1.9603 & 1.8722 \\
\midrule
\multicolumn{6}{c}{$n=5000$} \\
\midrule
Sine & $0.1$ & 44.5285 & 44.4750 & 44.4485 & 44.4327 \\
 & $0.5$ & 9.6353 & 9.3847 & 9.2581 & 9.1823 \\
 & $1$ & 5.7651 & 5.3358 & 5.1098 & 4.9711 \\
 & $5$ & 3.5091 & 2.8181 & 2.4338 & 2.1960 \\
 & $\infty$ & 1.6642 & 1.3290 & 1.2028 & 1.1516 \\
\addlinespace[2pt]
Triangular & $0.1$ & 44.5252 & 44.4716 & 44.4451 & 44.4293 \\
 & $0.5$ & 9.6271 & 9.3763 & 9.2496 & 9.1737 \\
 & $1$ & 5.7547 & 5.3245 & 5.0981 & 4.9590 \\
 & $5$ & 3.4654 & 2.7636 & 2.3706 & 2.1257 \\
 & $\infty$ & 1.6036 & 1.2450 & 1.1061 & 1.0475 \\
\addlinespace[2pt]
Exponentiated Brownian & $0.1$ & 44.5160 & 44.4624 & 44.4359 & 44.4201 \\
 & $0.5$ & 9.6147 & 9.3636 & 9.2367 & 9.1607 \\
 & $1$ & 5.7354 & 5.3037 & 5.0764 & 4.9367 \\
 & $5$ & 3.5393 & 2.8556 & 2.4770 & 2.2438 \\
 & $\infty$ & 1.7357 & 1.4245 & 1.3097 & 1.2626 \\
\midrule
\multicolumn{6}{c}{$n=20000$} \\
\midrule
Sine & $0.1$ & 23.4829 & 23.3812 & 23.3307 & 23.3007 \\
 & $0.5$ & 5.8907 & 5.4712 & 5.2511 & 5.1162 \\
 & $1$ & 4.2939 & 3.6977 & 3.3635 & 3.1488 \\
 & $5$ & 3.2862 & 2.5354 & 2.1001 & 1.8194 \\
 & $\infty$ & 1.2440 & 0.9646 & 0.8713 & 0.8436 \\
\addlinespace[2pt]
Triangular & $0.1$ & 23.4702 & 23.3685 & 23.3179 & 23.2879 \\
 & $0.5$ & 5.8812 & 5.4611 & 5.2405 & 5.1053 \\
 & $1$ & 4.2815 & 3.6833 & 3.3476 & 3.1319 \\
 & $5$ & 3.2474 & 2.4849 & 2.0390 & 1.7485 \\
 & $\infty$ & 1.2160 & 0.9274 & 0.8298 & 0.8001 \\
\addlinespace[2pt]
Exponentiated Brownian & $0.1$ & 23.4746 & 23.3729 & 23.3223 & 23.2923 \\
 & $0.5$ & 5.8634 & 5.4419 & 5.2205 & 5.0849 \\
 & $1$ & 4.2541 & 3.6515 & 3.3127 & 3.0945 \\
 & $5$ & 3.3146 & 2.5721 & 2.1442 & 1.8701 \\
 & $\infty$ & 1.1935 & 0.8955 & 0.7926 & 0.7575 \\
\bottomrule
\end{tabular}
\end{table*}

\end{document}